%% file: Duplij_PolSigma-arxiv.tex
\documentclass[12pt,reqno,twoside]{amsart}%
\usepackage{amsmath,amsthm,amscd,amsthm,upref,indentfirst}
\usepackage{amsfonts,mathrsfs}
\usepackage{amssymb,amsbsy,bm}
\usepackage{graphicx}

\usepackage{amsaddr}

\usepackage{soul}

\usepackage[us,long,24hr]{datetime}

\usepackage{mathabx}
\usepackage[usenames,dvipsnames]{color}

\theoremstyle{plain}
\newtheorem{theorem}{Theorem}
\newtheorem{assertion}[theorem]{Assertion}

\newtheorem{proposition}[theorem]{Proposition}

\theoremstyle{definition}
\newtheorem{definition}[theorem]{Definition}

\theoremstyle{remark}
\newtheorem{remark}[theorem]{Remark}

\newtheorem{example}[theorem]{Example}
\numberwithin{equation}{section}
\numberwithin{theorem}{section}

\normalfont\upshape
\usepackage{exscale}
\usepackage[scaled=0.86]{helvet}

\renewcommand{\mathit}{\bm}

\renewcommand{\mathtt}[1]{\scalebox{1.2}{\bf \texttt{\upshape#1}}}
\renewcommand{\emph}[1]{\textcolor{blue}{\textbf{#1}}}

\usepackage{float}

\usepackage[justification=centering]{caption}

\numberwithin{equation}{section}
\numberwithin{theorem}{section}

\usepackage[normalem]{ulem}
\usepackage[square,comma]{natbib}

\usepackage{mparhack}

\usepackage{keyval}
\usepackage{totcount}
\newtotcounter{citnum} 
\def\oldbibitem{} \let\oldbibitem=\bibitem
\def\bibitem{\stepcounter{citnum}\oldbibitem}

\usepackage[verbose]{backref}
\backrefsetup{verbose=false}
\renewcommand*{\backref}[1]{}
\renewcommand*{\backrefalt}[4]{[{\tiny%
    \ifcase #1 \textsl{Not cited}%
          \or \textsl{Cited on page}~\textcolor{BrickRed}{#2}%
          \else \textsl{Cited on pages}~\textcolor{BrickRed}{#2}%
    \fi%
    }]}

\usepackage{epigraph}

\usepackage{fancyhdr}
\author{\small\scshape S\lowercase{teven} D\lowercase{uplij}}

\address{%
University of M\"unster,
D-48149 M\"unster,
Germany}
\email{\small \sf douplii@uni-muenster.de;
sduplij@gmail.com;
http://www.uni-muenster.de/IT.StepanDouplii}

\title{\large\bfseries\scshape
P\lowercase{olyadic sigma matrices}}

\date{\textit{of start} February 12, 2024. \textit{Date}:
\textit{of completion}
March 27, 2024.
\newline
\mbox{}\hskip 1.16em
\textit{Total}:
15
references
}

\renewcommand{\refname}{\textsc{References}}

\let\origsection\section
\renewcommand{\section}[1]{\sectionmark{#1}\origsection{#1}}
\let\origsubsection\subsection
\renewcommand{\subsection}[1]{\subsectionmark{#1}\origsubsection{#1}}

\makeatletter
\renewenvironment{thebibliography}[1]{%
  \@xp\origsection\@xp*\@xp{\refname}%
  \normalfont\footnotesize\labelsep .9em\relax
  \renewcommand\theenumiv{\arabic{enumiv}}\let\p@enumiv\@empty
  \vspace*{-5pt}
  \list{\@biblabel{\theenumiv}}{\settowidth\labelwidth{\@biblabel{#1}}%
    \leftmargin\labelwidth \advance\leftmargin\labelsep
    \usecounter{enumiv}}%
  \sloppy \clubpenalty\@M \widowpenalty\clubpenalty
  \sfcode`\.=\@m
}{%
  \def\@noitemerr{\@latex@warning{Empty `thebibliography' environment}}%
  \endlist
}
\makeatother

\subjclass[2010]{20B05, 20H20, 20N10, 20N15, 20M10}
\keywords{sigma matrix, Pauli matrix, Pauli group, arity, $n$-ary semigroup, polyadic group, ternary group, $n$-ary group, querelement, finite semigroup, finite group, cyclic group, cyclic shift matrix, Hadamard product}

\begin{document}
\mbox{}
\vskip 2cm
\begin{abstract}

\input{
Duplij_PolSigma-abs.tex}

\end{abstract}

\maketitle

\thispagestyle{empty}
\mbox{}
\vspace{-1.4cm}
\tableofcontents
\newpage

\pagestyle{fancy}

\addtolength{\footskip}{15pt}

\renewcommand{\sectionmark}[1]{%
\markboth{
{ \scshape #1}}{}}

\renewcommand{\subsectionmark}[1]{%
\markright{
\mbox{\;}\\[5pt]
\textmd{#1}}{}}

\fancyhead{}
\fancyhead[EL,OR]{\leftmark}
\fancyhead[ER,OL]{\rightmark}
\fancyfoot[C]{\scshape -- \textcolor{BrickRed}{\thepage} --}

\renewcommand\headrulewidth{0.5pt}
\fancypagestyle {plain1}{ %
\fancyhf{}
\renewcommand {\headrulewidth }{0pt}
\renewcommand {\footrulewidth }{0pt}
}

\fancypagestyle{plain}{ %
\fancyhf{}
\fancyhead[C]{\scshape S\lowercase{teven} D\lowercase{uplij} \hskip 0.7cm \MakeUppercase{Polyadic Hopf algebras and quantum groups}}
\fancyfoot[C]{\scshape - \thepage  -}
\renewcommand {\headrulewidth }{0pt}
\renewcommand {\footrulewidth }{0pt}
}

\fancypagestyle{fancyref}{ %
\fancyhf{} 
\fancyhead[C]{\scshape R\lowercase{eferences} }
\fancyfoot[C]{\scshape -- \textcolor{BrickRed}{\thepage} --}
\renewcommand {\headrulewidth }{0.5pt}
\renewcommand {\footrulewidth }{0pt}
}

\fancypagestyle{emptyf}{
\fancyhead{}
\fancyfoot[C]{\scshape -- \textcolor{BrickRed}{\thepage} --}
\renewcommand{\headrulewidth}{0pt}
}
\thispagestyle{emptyf}
\input{Duplij_PolSigma-sw}

\pagestyle{emptyf}
\mbox{}

\input{Duplij_PolSigma.bbl}
\end{document}

%% file: Duplij_PolSigma-abs.tex
\noindent We generalize $\sigma$-matrices to higher arities using the
polyadization procedure proposed by the author. We build the nonderived
$n$-ary version of $SU\left(  2\right)  $ using cyclic shift block matrices.
We define a new function, the polyadic trace, which has an additivity property
analogous to the ordinary trace for block diagonal matrices and which can be
used to build the corresponding invariants. The elementary $\Sigma$-matrices
introduced here play a role similar to ordinary matrix units, and their sums
are full $\Sigma$-matrices which can be treated as a polyadic analog of
$\sigma$-matrices. The presentation of $n$-ary $SU\left(  2\right)  $ in terms
of full $\Sigma$-matrices is done using the Hadamard product. We then
generalize the Pauli group in two ways: for the binary case we introduce the
extended phase shifted $\sigma$-matrices with multipliers in cyclic groups of
order $4q$ ($q>4$), and for the polyadic case we construct the correspondent
finite $n$-ary semigroup of phase-shifted elementary $\Sigma$-matrices of
order $4q\left(  n-1\right)  +1$, and the finite $n$-ary group of
phase-shifted full $\Sigma$-matrices of order $4q$. Finally, we introduce the
finite $n$-ary group of heterogeneous full $\mathit{\Sigma}^{het}$-matrices of
order $\left(  4q\left(  n-1\right)  \right)  ^{4}$. Some examples of the
lowest arities are presented.

%% file: Duplij_PolSigma-sw.tex

\section{\textsc{Introduction}}

The role of sigma matrices ($\sigma$-matrices, Pauli matrices) in mathematical
physics is hard to overestimate. For a review, see, e.g.
\cite{fey/lei/san,schiff,liboff}. Here we will generalize $\sigma$-matrices to
higher arities by using the polyadization procedure proposed in
\cite{duplij2022}.

First, we recall the connection between $\sigma$-matrices and $SU\left(
2\right)  $, and then construct an $n$-ary version of them using cyclic block
shift matrices along the lines of \cite{duplij2022}. We define a new function,
a polyadic trace, which has an additivity property analogous to the ordinary
trace for block diagonal matrices, and can be used to build the corresponding
invariants. The elementary $\Sigma$-matrices introduced here play a role
similar to ordinary matrix units, and their sums are full $\Sigma$-matrices
which can be treated as a polyadic analog of ordinary (binary) $\sigma
$-matrices. The presentation of $n$-ary $SU\left(  2\right)  $ in terms of
full $\Sigma$-matrices is done by means of the Hadamard product.

Second, we consider generalizations of the Pauli group (see, e.g.
\cite{nie/chu,kib2009}): for the binary case we introduce extended phase
shifted $\sigma$-matrices with multipliers in cyclic groups of order $4q$
($q>4$), for the polyadic case we construct the corresponding finite $n$-ary
semigroup of phase-shifted elementary $\Sigma$-matrices of order $4q\left(
n-1\right)  +1$, and a finite $n$-ary group of phase-shifted full $\Sigma
$-matrices of order $4q$.

Finally, we introduce the finite $n$-ary group of heterogeneous full
$\mathit{\Sigma}^{het}$-matrices of order $\left(  4q\left(  n-1\right)
\right)  ^{4}$, and some illustrative examples in the lowest arities are presented.

\section{\textsc{Preliminaries}}

First, to establish notation we briefly recall some well-known facts in
language convenient for further generalizations to higher arity. Consider the
special unitary group $SU\left(  2\right)  $ represented by $2\times2$ complex
($3$-parameter) matrices of the form%

\begin{align}
M  &  =M(x_{0},\overrightarrow{x})=\left(
\begin{array}
[c]{cc}%
x_{0}+ix_{1} & x_{2}+ix_{3}\\
-x_{2}+x_{3}i & x_{0}-x_{1}i
\end{array}
\right)  ,\label{m}\\
\det M  &  =x_{0}^{2}+\overrightarrow{x}^{2}=1,\ \ \ \ x_{0},x_{1},x_{2}%
,x_{3}\in\mathbb{R}, \label{m1}%
\end{align}
such that%
\begin{equation}
M^{\dagger}M=MM^{\dagger}=I_{2}, \label{u}%
\end{equation}
where $\left(  \dagger\right)  $ is the Hermitian (conjugate) transpose (here
$M^{\dagger}=M^{-1}$), and $\overrightarrow{x}=\left(  x_{1},x_{2}%
,x_{3}\right)  $.

The binary product of the matrices (\ref{m}) is closed and in terms of the
parameters is%
\begin{align}
&  M(x_{0}^{\prime},\overrightarrow{x}^{\prime})M(x_{0}^{\prime\prime
},\overrightarrow{x}^{\prime\prime})=M(x_{0},\overrightarrow{x}),\\
&  x_{0}^{\prime2}+\overrightarrow{x}^{\prime2}=x_{0}^{\prime\prime
2}+\overrightarrow{x}^{\prime\prime2}=1,\ \ \ \ x_{0}^{\prime},x_{1}^{\prime
},x_{2}^{\prime},x_{3}^{\prime},x_{0}^{\prime\prime},x_{1}^{\prime\prime
},x_{2}^{\prime\prime},x_{3}^{\prime\prime}\in\mathbb{R},
\end{align}
where%
\begin{align}
x_{0}  &  =x_{0}^{\prime}x_{0}^{\prime\prime}-x_{1}^{\prime}x_{1}%
^{\prime\prime}-x_{2}^{\prime}x_{2}^{\prime\prime}-x_{3}^{\prime}x_{3}%
^{\prime\prime},\ \ \ x_{1}=x_{1}^{\prime}x_{0}^{\prime\prime}+x_{0}^{\prime
}x_{1}^{\prime\prime}+x_{2}^{\prime}x_{3}^{\prime\prime}-x_{3}^{\prime}%
x_{2}^{\prime\prime},\label{a}\\
x_{2}  &  =x_{2}^{\prime}x_{0}^{\prime\prime}+x_{0}^{\prime}x_{2}%
^{\prime\prime}+x_{3}^{\prime}x_{1}^{\prime\prime}-x_{1}^{\prime}x_{3}%
^{\prime\prime},\ \ \ x_{3}=x_{3}^{\prime}x_{0}^{\prime\prime}+x_{0}^{\prime
}x_{3}^{\prime\prime}+x_{1}^{\prime}x_{2}^{\prime\prime}-x_{2}^{\prime}%
x_{1}^{\prime\prime}. \label{d}%
\end{align}

The standard (binary) $\sigma$-matrices (Pauli matrices)%
\begin{equation}
\sigma_{0}=\left(
\begin{array}
[c]{cc}%
1 & 0\\
0 & 1
\end{array}
\right)  \equiv I_{2}\equiv\mathsf{1},\ \ \sigma_{1}=\left(
\begin{array}
[c]{cc}%
0 & 1\\
1 & 0
\end{array}
\right)  ,\ \ \sigma_{2}=\left(
\begin{array}
[c]{cc}%
0 & -i\\
i & 0
\end{array}
\right)  ,\ \ \sigma_{3}=\left(
\begin{array}
[c]{cc}%
1 & 0\\
0 & -1
\end{array}
\right)  \label{si}%
\end{equation}
are involutory (reflections)%
\begin{equation}
\sigma_{j}^{2}=I_{2},\ \ \ \ \ \ \sigma_{j}^{-1}=\sigma_{j,}\label{in}%
\end{equation}
and they satisfy the (binary) commutation and anti-commutation relations%
\begin{align}
\left[  \sigma_{j},\sigma_{k}\right]   &  =\left[  \sigma_{j},\sigma
_{k}\right]  ^{\left[  \mathbf{2}\right]  }=\sigma_{j}\sigma_{k}-\sigma
_{k}\sigma_{j}=2i\epsilon_{jkl}\sigma_{l},\ \ \ j,k,l=1,2,3,\label{c2}\\[1pt]
\left\{  \sigma_{j},\sigma_{k}\right\}   &  =\left\{  \sigma_{j},\sigma
_{k}\right\}  ^{\left[  \mathbf{2}\right]  }=\sigma_{j}\sigma_{k}+\sigma
_{k}\sigma_{j}=2\delta_{jk}I_{2},\label{c2a}%
\end{align}
where $\epsilon_{jkl}$ is the Levi-Civita (permutation) symbol in three
dimensions, while $\left[  \ \ ,\ \ \right]  ^{\left[  \mathbf{2}\right]  }$
and $\left\{  \ \ ,\ \ \right\}  ^{\left[  \mathbf{2}\right]  }$ are the
ordinary (binary) commutator and anti-commutator (here we denote arity by bold
numbers in square brackets).

In terms of $\sigma$-matrices the group element of $SU\left(  2\right)  $ has
the following presentation (see (\ref{m}))%
\begin{equation}
M=M(x_{0},\overrightarrow{x})=x_{0}\sigma_{0}+i\overrightarrow{x}%
\overrightarrow{\sigma}. \label{ms}%
\end{equation}

The binary $SU\left(  2\right)  $ invariant $I_{gen}$ can be obtained using
the ordinary trace in the standard way%
\begin{equation}
I_{gen}=\frac{1}{2}\operatorname*{tr}\left(  M^{\dagger}(x_{0}^{\prime
},\overrightarrow{x}^{\prime})M(x_{0}^{\prime\prime},\overrightarrow
{x}^{\prime\prime})\right)  =x_{0}^{\prime}x_{0}^{\prime\prime}%
+\overrightarrow{x}^{\prime}\overrightarrow{x}^{\prime\prime}. \label{i2}%
\end{equation}
$\allowbreak$

The invariant in the form (\ref{i2}) will allow us to obtain its analog for
higher arities.

\section{\textsc{Polyadic} $SU\left(  2\right)  $}

Let us apply the matrix polyadization procedure to $SU\left(  2\right)  $
(\ref{m}) following \cite{duplij2022}. Consider the set of $\left(
n-1\right)  $ complex $2\times2$ matrices
\begin{align}
M^{\left(  \mathbf{k}\right)  }  &  =M(x_{0}^{\left(  \mathbf{k}\right)
},\overrightarrow{x}^{\left(  \mathbf{k}\right)  })=\left(
\begin{array}
[c]{cc}%
x_{0}^{\left(  \mathbf{k}\right)  }+ix_{3}^{\left(  \mathbf{k}\right)  } &
ix_{1}^{\left(  \mathbf{k}\right)  }+x_{2}^{\left(  \mathbf{k}\right)  }\\
ix_{1}^{\left(  \mathbf{k}\right)  }-x_{2}^{\left(  \mathbf{k}\right)  } &
x_{0}^{\left(  \mathbf{k}\right)  }-ix_{3}^{\left(  \mathbf{k}\right)  }%
\end{array}
\right)  =x_{0}^{\left(  \mathbf{k}\right)  }\sigma_{0}+i\overrightarrow
{x}^{\left(  \mathbf{k}\right)  }\overrightarrow{\sigma},\label{mk}\\
M^{\left(  \mathbf{k}\right)  \dagger}M^{\left(  \mathbf{k}\right)  }  &
=\mathsf{1},\ \ \ \ \det M^{\left(  \mathbf{k}\right)  }=\left(
x_{0}^{\left(  \mathbf{k}\right)  }\right)  ^{2}+\left(  \overrightarrow
{x}^{\left(  \mathbf{k}\right)  }\right)  ^{2}=1,\ \ k=1,\ldots n-1,
\label{mk1}%
\end{align}
such that each one represents the group $SU\left(  2\right)  $ (see
(\ref{m})). For clarity, we denote the $2\times2$ identity and zero matrices
by $\mathsf{1}$ and $\mathsf{0}$.

\begin{definition}
The polyadic ($n$-ary) special linear $3\left(  n-1\right)  $ parameter group
$SU^{\left[  \mathbf{n}\right]  }\left(  2\right)  $ is isomorphic to the set
of $2\left(  n-1\right)  \times2\left(  n-1\right)  $ cyclic shift block
matrices of the form \cite{duplij2022}%
\begin{equation}
\mathbf{M}=\mathbf{M}_{2\left(  n-1\right)  \times2\left(  n-1\right)
}=\left(
\begin{array}
[c]{ccccc}%
\mathsf{0} & M^{\left(  \mathbf{1}\right)  } & \ldots & \mathsf{0} &
\mathsf{0}\\
\mathsf{0} & \mathsf{0} & M^{\left(  \mathbf{2}\right)  } & \ldots &
\mathsf{0}\\
\mathsf{0} & \mathsf{0} & \ddots & \ddots & \vdots\\
\vdots & \vdots & \ddots & \mathsf{0} & M^{\left(  \mathbf{n-2}\right)  }\\
M^{\left(  \mathbf{n-1}\right)  } & \mathsf{0} & \ldots & \mathsf{0} &
\mathsf{0}%
\end{array}
\right)  , \label{mn}%
\end{equation}
with respect to the $n$-ary multiplication (as the ordinary $2\left(
n-1\right)  \times2\left(  n-1\right)  $-matrix product $\left(  \cdot\right)
$)%
\begin{equation}
\mathit{\mu}^{\left[  \mathbf{n}\right]  }\left[  \overset{n}{\overbrace
{\mathbf{M}^{\prime},\mathbf{M}^{\prime\prime},\ldots,\mathbf{M}^{\prime
\prime\prime}}}\right]  =\overset{n}{\overbrace{\mathbf{M}^{\prime}%
\cdot\mathbf{M}^{\prime\prime}\ldots\cdot\mathbf{M}^{\prime\prime\prime}}%
}=\mathbf{M}, \label{mnm}%
\end{equation}
which is non-derived (i.e. not closed for fewer than $n$ multipliers). The
allowed number of $\mathbf{M}$'s in any polyadic product is $\ell\left(
n-1\right)  +1$, $\ell\in\mathbb{N}$.
\end{definition}

Sometimes, we omit $\mathit{\mu}^{\left[  \mathbf{n}\right]  }$, if the arity
is clear from the context.

In terms of $SU\left(  2\right)  $ blocks $M^{\left(  \mathbf{k}\right)  }$
(\ref{mn}) the product (\ref{mnm}) is given by the $\left(  n-1\right)  $
cycled products%
\begin{align}
\overset{n}{\overbrace{M^{\left(  \mathbf{1}\right)  \prime}M^{\left(
\mathbf{2}\right)  \prime\prime}\ldots M^{\left(  \mathbf{n-1}\right)
\prime\prime\prime}M^{\left(  \mathbf{1}\right)  \prime\prime\prime\prime}}}
&  =M^{\left(  \mathbf{1}\right)  },\label{mm1}\\
\overset{n}{\overbrace{M^{\left(  \mathbf{2}\right)  \prime}M^{\left(
\mathbf{3}\right)  \prime\prime}\ldots M^{\left(  \mathbf{1}\right)
\prime\prime\prime}M^{\left(  \mathbf{2}\right)  \prime\prime\prime\prime}}}
&  =M^{\left(  \mathbf{2}\right)  },\\
&  \vdots\\
\overset{n}{\overbrace{M^{\left(  \mathbf{n-1}\right)  \prime}M^{\left(
\mathbf{1}\right)  \prime\prime}\ldots M^{\left(  \mathbf{n-2}\right)
\prime\prime\prime}M^{\left(  \mathbf{n-1}\right)  \prime\prime\prime\prime}%
}}  &  =M^{\left(  \mathbf{n-1}\right)  }. \label{mmn}%
\end{align}

The $n$-ary analog of the binary inverse is the querelement $\widetilde
{\mathbf{M}}$ \cite{dor3} defined by%
\begin{equation}
\mathit{\mu}^{\left[  \mathbf{n}\right]  }\left[  \overset{n}{\overbrace
{\mathbf{M},\mathbf{M},\ldots,\widetilde{\mathbf{M}}}}\right]  =\mathbf{M,}
\label{mq}%
\end{equation}
where $\widetilde{\mathbf{M}}$ can be on any place. The manifest form of
$\widetilde{\mathbf{M}}$ can be obtained from (\ref{mm1})--(\ref{mq}) as
(taking into account that $M^{-1}=M^{\dagger}$, $M\in SU\left(  2\right)
$){\tiny
\begin{equation}
\widetilde{\mathbf{M}}=\left(
\begin{array}
[c]{ccccc}%
\mathsf{0} & M^{\left(  \mathbf{n-1}\right)  \dagger}\ldots M^{\left(
\mathbf{3}\right)  \dagger}M^{\left(  \mathbf{2}\right)  \dagger} & \ldots &
\mathsf{0} & \mathsf{0}\\
\mathsf{0} & \mathsf{0} & M^{\left(  \mathbf{1}\right)  \dagger}\ldots
M^{\left(  \mathbf{4}\right)  \dagger}M^{\left(  \mathbf{3}\right)  \dagger} &
\ldots & \mathsf{0}\\
\mathsf{0} & \mathsf{0} & \ddots & \ddots & \vdots\\
\vdots & \vdots & \ddots & \mathsf{0} & M^{\left(  \mathbf{n-3}\right)
\dagger}\ldots M^{\left(  \mathbf{1}\right)  \dagger}M^{\left(  \mathbf{n-1}%
\right)  \dagger}\\
M^{\left(  \mathbf{n-2}\right)  \dagger}\ldots M^{\left(  \mathbf{2}\right)
\dagger}M^{\left(  \mathbf{1}\right)  \dagger} & \mathsf{0} & \ldots &
\mathsf{0} & \mathsf{0}%
\end{array}
\right)  . \label{mw}%
\end{equation}
}

Thus, we have

\begin{definition}
The set of matrices (\ref{mn}) together with the $n$-ary product (\ref{mnm})
and querelement (\ref{mw}) defines a polyadic ($n$-ary) special unitary group%
\begin{equation}
SU^{\left[  \mathbf{n}\right]  }\left(  2\right)  =\left\langle \left\{
\mathbf{M}\right\}  \mid\mathit{\mu}^{\left[  \mathbf{n}\right]  }%
,\widetilde{\left(  \ \ \right)  }\right\rangle .
\end{equation}

\end{definition}

The existence of the polyadic identity in an $n$-ary group is not necessary,
as only the querelement is important in this context \cite{dor3}. For the
$n$-ary group $SU^{\left[  \mathbf{n}\right]  }\left(  2\right)  $ the
situation is the opposite: the polyadic identity is not unique. Indeed, recall
that the left $\mathbf{E}_{l}$, middle $\mathbf{E}_{m}$ (on $n-2$ places) and
right $\mathbf{E}_{r}$ polyadic identities are defined by%
\begin{align}
\mathit{\mu}^{\left[  \mathbf{n}\right]  }\left[  \overset{n}{\overbrace
{\mathbf{E}_{l},\mathbf{M},\ldots,\mathbf{M}}}\right]   &  =\mathbf{M}%
,\label{el}\\
\mathit{\mu}^{\left[  \mathbf{n}\right]  }\left[  \overset{n}{\overbrace
{\mathbf{M},\mathbf{E}_{m},\ldots,\mathbf{M}}}\right]   &  =\mathbf{M}%
,\label{em}\\
\mathit{\mu}^{\left[  \mathbf{n}\right]  }\left[  \overset{n}{\overbrace
{\mathbf{M},\ldots,\mathbf{M},\mathbf{E}_{r}}}\right]   &  =\mathbf{M}%
,\ \ \ \forall\mathbf{M}\in SU^{\left[  \mathbf{n}\right]  }\left(  2\right)
, \label{er}%
\end{align}
respectively.

\begin{assertion}
The $n$-ary group $SU^{\left[  \mathbf{n}\right]  }\left(  2\right)  $ has
sets of left and right identities of the form%
\begin{align}
\mathbf{E}_{l}\left(  a\right)   &  =\left(
\begin{array}
[c]{ccccc}%
\mathsf{0} & a^{\left(  \mathbf{1}\right)  }I_{2} & \ldots & \mathsf{0} &
\mathsf{0}\\
\mathsf{0} & \mathsf{0} & a^{\left(  \mathbf{2}\right)  }I_{2} & \ldots &
\mathsf{0}\\
\mathsf{0} & \mathsf{0} & \ddots & \ddots & \vdots\\
\vdots & \vdots & \ddots & \mathsf{0} & a^{\left(  \mathbf{n-2}\right)  }%
I_{2}\\
a^{\left(  \mathbf{n-1}\right)  }I_{2} & \mathsf{0} & \ldots & \mathsf{0} &
\mathsf{0}%
\end{array}
\right)  ,\label{el1}\\
a^{\left(  \mathbf{1}\right)  }a^{\left(  \mathbf{2}\right)  }\ldots
a^{\left(  \mathbf{n-1}\right)  }  &  =1,\ a^{\left(  \mathbf{k}\right)  }%
\in\mathbb{R}\setminus\left\{  0\right\}  ,\label{el2}\\
\mathbf{E}_{r}\left(  b\right)   &  =\left(
\begin{array}
[c]{ccccc}%
\mathsf{0} & b^{\left(  \mathbf{1}\right)  }I_{2} & \ldots & \mathsf{0} &
\mathsf{0}\\
\mathsf{0} & \mathsf{0} & b^{\left(  \mathbf{2}\right)  }I_{2} & \ldots &
\mathsf{0}\\
\mathsf{0} & \mathsf{0} & \ddots & \ddots & \vdots\\
\vdots & \vdots & \ddots & \mathsf{0} & b^{\left(  \mathbf{n-2}\right)  }%
I_{2}\\
b^{\left(  \mathbf{n-1}\right)  }I_{2} & \mathsf{0} & \ldots & \mathsf{0} &
\mathsf{0}%
\end{array}
\right)  ,\label{er1}\\
b^{\left(  \mathbf{1}\right)  }b^{\left(  \mathbf{2}\right)  }\ldots
b^{\left(  \mathbf{n-1}\right)  }  &  =1,\ b^{\left(  \mathbf{k}\right)  }%
\in\mathbb{R}\setminus\left\{  0\right\}  , \label{er2}%
\end{align}
where both $a^{\left(  \mathbf{k}\right)  }$ and $b^{\left(  \mathbf{k}%
\right)  }$, $k=1,\ldots,n-1$, are different sets of $\left(  n-2\right)  $
non-zero reals (satisfying the additional conditions (\ref{el2}) and
(\ref{er2})), and there are no middle identities at all. They form two
different $\left(  n-2\right)  $ parameter subgroups of $SU^{\left[
\mathbf{n}\right]  }\left(  2\right)  $.
\end{assertion}

\begin{remark}
In the case when all $a^{\left(  \mathbf{k}\right)  }=b^{\left(
\mathbf{k}\right)  }=1$, both identities coincide%
\begin{equation}
\mathbf{E}_{l}\left(  1\right)  =\mathbf{E}_{r}\left(  1\right)
=\mathbf{E}^{\left[  \mathbf{n}\right]  }=\left(
\begin{array}
[c]{ccccc}%
\mathsf{0} & I_{2} & \ldots & \mathsf{0} & \mathsf{0}\\
\mathsf{0} & \mathsf{0} & I_{2} & \ldots & \mathsf{0}\\
\mathsf{0} & \mathsf{0} & \ddots & \ddots & \vdots\\
\vdots & \vdots & \ddots & \mathsf{0} & I_{2}\\
I_{2} & \mathsf{0} & \ldots & \mathsf{0} & \mathsf{0}%
\end{array}
\right)  , \label{ee}%
\end{equation}
and when $\mathbf{E}$ is in one of the middle places in (\ref{em}), it
permutes the internal blocks of (\ref{mn}), and therefore $\mathbf{E}$ is not
a middle polyadic identity. Nevertheless, if all the blocks are equal
$M^{\left(  \mathbf{k}\right)  }=M$, then $\mathbf{E}$ is also a middle
identity, or simply the polyadic identity of $SU^{\left[  \mathbf{n}\right]
}\left(  2\right)  $.
\end{remark}

For the cyclic shift block matrices of the form (\ref{mn}), the determinant is%
\begin{equation}
\det\mathbf{M}=\left(  -1\right)  ^{n-1}\det M^{\left(  \mathbf{1}\right)
}\ldots\det M^{\left(  \mathbf{n-1}\right)  }, \label{dm}%
\end{equation}
which for $M^{\left(  \mathbf{k}\right)  }\in SU\left(  2\right)  $ becomes
$\det\mathbf{M}=\left(  -1\right)  ^{n-1}$.

Consider a $m\left(  n-1\right)  \times m\left(  n-1\right)  $ block-diagonal
matrix $\mathbf{A}=\mathrm{diag}\left(  A_{1}\ldots A_{n-1}\right)  $, where
$A_{k}$ are $\left(  n-1\right)  $ matrices of the size $m\times m$. The
connection between trace and determinant of any matrix $\mathbf{A}$ is%
\begin{equation}
\det e^{\mathbf{A}}=e^{\operatorname*{tr}\mathbf{A}}, \label{da}%
\end{equation}
which in case of the block-diagonal matrix becomes%
\begin{equation}
\det e^{\mathbf{A}}=e^{\operatorname*{tr}A_{1}}\ldots e^{\operatorname*{tr}%
A_{n-1}}=e^{\operatorname*{tr}A_{1}+\ldots+\operatorname*{tr}A_{n-1}},
\label{de}%
\end{equation}
and for the block-diagonal matrix $\mathbf{A}$ we have%
\begin{align}
\det\mathbf{A}  &  =\det A_{1}\ldots\det A_{n-1}\label{da1}\\
\operatorname*{tr}\mathbf{A}  &  =\operatorname*{tr}A_{1}+\ldots
+\operatorname*{tr}A_{n-1} \label{ta1}%
\end{align}

Note that the ordinary trace of the cyclic shift block matrix (\ref{mn}) is
identically zero $\operatorname*{tr}\mathbf{M}=0$. Comparing (\ref{da1}) with
(\ref{dm}) we observe that both cyclic shift block matrix $\mathbf{M}$ and the
block-diagonal matrix $\mathbf{A}$ are monomial block matrices. By analogy
with (\ref{ta1}), we propose a new function

\begin{definition}
A polyadic (or $n$-ary) trace $\mathbf{Trp}^{\left[  n\right]  }:SU^{\left[
n\right]  }\left(  2\right)  \rightarrow\mathbb{R}$ is defined for the
$2\left(  n-1\right)  \times2\left(  n-1\right)  $ cyclic shift block matrix
(\ref{mn}) by%
\begin{equation}
\mathbf{Trp}^{\left[  n\right]  }\mathbf{M}=\operatorname*{tr}M^{\left(
\mathbf{1}\right)  }+\ldots+\operatorname*{tr}M^{\left(  \mathbf{n-1}\right)
}. \label{tn}%
\end{equation}

\end{definition}

\begin{remark}
In this case we see that the invariance properties of the polyadic trace
remain the same as for the ordinary trace for a block-diagonal matrix. This is
true for any monomial block matrix, for which we can use (\ref{ta1}),
(\ref{tn}).
\end{remark}

\begin{remark}
In the binary case ($n=2$) we have $\mathbf{M}=M^{\left(  \mathbf{1}\right)
}$, and the polyadic trace coincides with the ordinary trace%
\begin{equation}
\mathbf{Trp}^{\left[  2\right]  }\mathbf{M}=\operatorname*{tr}M^{\left(
\mathbf{1}\right)  }.
\end{equation}

\end{remark}

\begin{example}
The polyadic identities (\ref{el1})--(\ref{er1}) are traceless
$\operatorname*{tr}\mathbf{E}_{l}\left(  a\right)  =\operatorname*{tr}%
\mathbf{E}_{r}\left(  b\right)  =0$, while their polyadic traces are different%
\begin{equation}
\mathbf{Trp}^{\left[  n\right]  }\mathbf{E}_{l}\left(  a\right)  =2\left(
a^{\left(  \mathbf{1}\right)  }+\ldots+a^{\left(  \mathbf{n-1}\right)
}\right)  ,\ \ \ \ \ \ \mathbf{Trp}^{\left[  n\right]  }\mathbf{E}_{r}\left(
b\right)  =2\left(  b^{\left(  \mathbf{1}\right)  }+\ldots+b^{\left(
\mathbf{n-1}\right)  }\right)  .
\end{equation}

\end{example}

\section{\textsc{Polyadic analog of} \textsc{sigma matrices}}

Next we present the $SU^{\left[  n\right]  }\left(  2\right)  $ matrices
(\ref{mn}) in a form similar to (\ref{ms}) using (\ref{mk}). We insert
(\ref{mk}) into (\ref{mn}) and get{\tiny
\begin{equation}
\mathbf{M}=\left(
\begin{array}
[c]{ccccc}%
\mathsf{0} & x_{0}^{\left(  \mathbf{1}\right)  }\sigma_{0}+i\overrightarrow
{x}^{\left(  \mathbf{1}\right)  }\overrightarrow{\sigma} & \ldots & \mathsf{0}
& \mathsf{0}\\
\mathsf{0} & \mathsf{0} & x_{0}^{\left(  \mathbf{2}\right)  }\sigma
_{0}+i\overrightarrow{x}^{\left(  \mathbf{2}\right)  }\overrightarrow{\sigma}
& \ldots & \mathsf{0}\\
\mathsf{0} & \mathsf{0} & \ddots & \ddots & \vdots\\
\vdots & \vdots & \ddots & \mathsf{0} & x_{0}^{\left(  \mathbf{n-2}\right)
}\sigma_{0}+i\overrightarrow{x}^{\left(  \mathbf{n-2}\right)  }\overrightarrow
{\sigma}\\
x_{0}^{\left(  \mathbf{n-1}\right)  }\sigma_{0}+i\overrightarrow{x}^{\left(
\mathbf{n-2}\right)  }\overrightarrow{\sigma} & \mathsf{0} & \ldots &
\mathsf{0} & \mathsf{0}%
\end{array}
\right)  . \label{mx}%
\end{equation}
}

The ordinary trace of $\mathbf{M}$ is obviously zero, while the polyadic trace
(\ref{tn}) is%
\begin{equation}
\mathbf{Trp}^{\left[  n\right]  }\mathbf{M}=2\left(  x_{0}^{\left(
\mathbf{1}\right)  }+\ldots+x_{0}^{\left(  \mathbf{n-1}\right)  }\right)  .
\end{equation}

\subsection{Elementary $\Sigma$-matrices}

To obtain an analog of the decomposition (\ref{ms}) for $\mathbf{M}$
(\ref{mn}), (\ref{mx}), we should use cyclic shift block matrices constructed
from $\sigma$-matrices.

\begin{definition}
The $4\left(  n-1\right)  $ polyadic elementary $\Sigma$-matrices of the size
$2\left(  n-1\right)  \times2\left(  n-1\right)  $ are defined as the cyclic
shift block matrices (\ref{mn}) containing only one non-zero $2\times2$ block,
in which a corresponding $\sigma$-matrix is placed%
\begin{equation}
\Sigma_{j}^{\left(  \mathbf{k}\right)  }=\left(
\begin{array}
[c]{ccccccc}%
\mathsf{0} & \fbox{\textsf{0\ \ \ }$^{\left(  \mathbf{1}\right)  }$} &
\mathsf{0} & \ldots & \mathsf{0} & \ldots & \mathsf{0}\\
\mathsf{0} & \mathsf{0} & \fbox{\textsf{0\ \ \ }$^{\left(  \mathbf{2}\right)
}$} & \ldots & \mathsf{0} & \ldots & \mathsf{0}\\
\mathsf{0} & \mathsf{0} & \mathsf{0} & \ddots & \vdots & \ldots & \vdots\\
\mathsf{0} & \mathsf{0} & \mathsf{0} & \ldots & \fbox{$\sigma_{j}%
$\ \ $^{\left(  \mathbf{k}\right)  }$} & \ldots & \mathsf{0}\\
\vdots & \vdots & \vdots & \vdots & \vdots & \ddots & \vdots\\
\mathsf{0} & \mathsf{0} & \mathsf{0} & \ldots & \mathsf{0} & \ldots &
\fbox{\textsf{0\ \ \ }$^{\left(  \mathbf{n-2}\right)  }$}\\
\fbox{\textsf{0\ \ \ }$^{\left(  \mathbf{n-1}\right)  }$} & \mathsf{0} &
\mathsf{0} & \ldots & \mathsf{0} & \ldots & \mathsf{0}%
\end{array}
\right)  , \label{sk}%
\end{equation}
where $\sigma_{j}$ is in $k$th block on the diagonal.
\end{definition}

\begin{theorem}
The $SU^{\left[  n\right]  }\left(  2\right)  $ matrix can be represented as
an expansion in the polyadic elementary $\Sigma$-matrices by%
\begin{align}
\mathbf{M}=x_{0}^{\left(  \mathbf{1}\right)  }\Sigma_{0}^{\left(
\mathbf{1}\right)  }+  &  i\overrightarrow{x}^{\left(  \mathbf{1}\right)
}\overrightarrow{\Sigma}^{\left(  \mathbf{1}\right)  }+x_{0}^{\left(
\mathbf{2}\right)  }\Sigma_{0}^{\left(  \mathbf{2}\right)  }+i\overrightarrow
{x}^{\left(  \mathbf{2}\right)  }\overrightarrow{\Sigma}^{\left(
\mathbf{2}\right)  }+\ldots+x_{0}^{\left(  \mathbf{n-1}\right)  }\Sigma
_{0}^{\left(  \mathbf{n-1}\right)  }+i\overrightarrow{x}^{\left(
\mathbf{n-1}\right)  }\overrightarrow{\Sigma}^{\left(  \mathbf{n-1}\right)
},\label{mxs}\\
&  \overrightarrow{\Sigma}^{\left(  \mathbf{k}\right)  }=\left(  \Sigma
_{1}^{\left(  \mathbf{k}\right)  },\Sigma_{2}^{\left(  \mathbf{k}\right)
},\Sigma_{3}^{\left(  \mathbf{k}\right)  }\right)  ,\ \ \left(  x_{0}^{\left(
\mathbf{k}\right)  }\right)  ^{2}+\left(  \overrightarrow{x}^{\left(
\mathbf{k}\right)  }\right)  ^{2}=1,\ \ k=1,\ldots n-1. \label{mx1}%
\end{align}

\end{theorem}

\begin{proof}
The presentation (\ref{mxs}) follows from the exact form (\ref{mx}), the
definition of the $\Sigma$-matrix (\ref{sk}), linearity and properties of the
scalar multiplication.
\end{proof}

\begin{proposition}
Each elementary $\Sigma$-matrix (\ref{sk}) is non-invertible, but the sum
(\ref{mxs}) gives an invertible $SU^{\left[  n\right]  }\left(  2\right)  $ matrix.
\end{proposition}

\begin{proof}
If one block in the cycled block shift matrix has zero determinant, the whole
matrix becomes non-invertible, because of the product of determinants
(\ref{dm}). Following the sphere $S_{3}$ condition (\ref{mx1}) for each block
$\left(  x_{0}^{\left(  \mathbf{k}\right)  }\right)  ^{2}+\left(
x_{1}^{\left(  \mathbf{k}\right)  }\right)  ^{2}+\left(  x_{2}^{\left(
\mathbf{k}\right)  }\right)  ^{2}+\left(  x_{3}^{\left(  \mathbf{k}\right)
}\right)  ^{2}=1$, at least one $x_{i}^{\left(  \mathbf{k}\right)  }\neq0$,
and so each block in (\ref{mx}) is invertible ($M^{\left(  \mathbf{k}\right)
}\in SU\left(  2\right)  $), then the whole matrix $\mathbf{M}$ is invertible.
\end{proof}

Another form of the polyadic elementary $\Sigma$-matrices can be given in
terms of a product of a vector-column and a vector-row. Indeed, introduce two
$\left(  n-1\right)  \times1$ vectors%
\begin{equation}
\mathbf{V}^{\left(  \mathbf{k}\right)  }=\left(  \left\vert
\begin{array}
[c]{c}%
\mathsf{0}\\
\vdots\\
\fbox{$I_{2}$\textsf{\ \ \ }$^{\left(  \mathbf{k}\right)  }$}\\
\vdots\\
\mathsf{0}%
\end{array}
\right\}  \left(  n-1\right)  \right)  ,\ \ \ \mathbf{S}_{j}^{\left(
\mathbf{k}\right)  }=\left(  \left\vert
\begin{array}
[c]{c}%
\mathsf{0}\\
\vdots\\
\fbox{$\sigma_{j}$\textsf{\ \ \ }$^{\left(  \mathbf{k}\right)  }$}\\
\vdots\\
\mathsf{0}%
\end{array}
\right\}  \left(  n-1\right)  \right)  ,\ \ \ \ j=0,1,2,3, \label{vs1}%
\end{equation}
then we obtain matrices with one nonzero block (block matrix units)%
\begin{equation}
\Sigma_{j}^{\left(  \mathbf{k}\right)  }=\mathbf{V}^{\left(  \mathbf{k}%
\right)  }\cdot\left(  \mathbf{S}_{j}^{\left(  \mathbf{k+1}\right)  }\right)
^{T},\ \ k=1,\ldots,n-2,\ \ \ \ \ \Sigma_{j}^{\left(  \mathbf{n-1}\right)
}=\mathbf{V}^{\left(  \mathbf{n-1}\right)  }\cdot\left(  \mathbf{S}%
_{j}^{\left(  \mathbf{1}\right)  }\right)  ^{T}, \label{sv}%
\end{equation}
which coincide with (\ref{sk}).

\begin{remark}
As opposed to the binary case and $\sigma$-matrices, the elementary $\Sigma
$-matrices (\ref{sk}) do not belong to $SU^{\left[  n\right]  }\left(
2\right)  $, because they are non-invertible, nilpotent and $\det\Sigma
_{j}^{\left(  \mathbf{k}\right)  }=0$.
\end{remark}

\subsection{Full $\mathit{\Sigma}$-matrices}

Note that from (\ref{ee}) and (\ref{sk}) it follows that the sum of elementary
$\Sigma$-matrices%
\begin{equation}
\Sigma_{0}^{\left(  \mathbf{1}\right)  }+\Sigma_{0}^{\left(  \mathbf{2}%
\right)  }+\ldots+\Sigma_{0}^{\left(  \mathbf{n-1}\right)  }=\mathbf{E\in
}SU^{\left[  n\right]  }\left(  2\right)  ,\label{se}%
\end{equation}
is invertible, because $\det\mathbf{E}=1$, and therefore, the sum is in
$SU^{\left[  n\right]  }\left(  2\right)  $. This allows us to make

\begin{definition}
The full $\Sigma$-matrices $\mathit{\Sigma}_{j}$ are defined by the sums of
elementary $\Sigma$-matrices%
\begin{equation}
\Sigma_{j}^{\left(  \mathbf{1}\right)  }+\Sigma_{j}^{\left(  \mathbf{2}%
\right)  }+\ldots+\Sigma_{j}^{\left(  \mathbf{n-1}\right)  }=\mathit{\Sigma
}_{j},\ \ j=0,1,2,3. \label{ssj}%
\end{equation}

\end{definition}

In manifest form the full $\mathit{\Sigma}$-matrices are%
\begin{equation}
\mathit{\Sigma}_{j}=\left(
\begin{array}
[c]{ccccccc}%
\mathsf{0} & \fbox{$\sigma_{j}$\textsf{\ \ \ }$^{\left(  \mathbf{1}\right)  }%
$} & \mathsf{0} & \ldots & \mathsf{0} & \ldots & \mathsf{0}\\
\mathsf{0} & \mathsf{0} & \fbox{$\sigma_{j}$\textsf{\ \ \ }$^{\left(
\mathbf{2}\right)  }$} & \ldots & \mathsf{0} & \ldots & \mathsf{0}\\
\mathsf{0} & \mathsf{0} & \mathsf{0} & \ddots & \vdots & \ldots & \vdots\\
\mathsf{0} & \mathsf{0} & \mathsf{0} & \ldots & \fbox{$\sigma_{j}%
$\ \ $^{\left(  \mathbf{k}\right)  }$} & \ldots & \mathsf{0}\\
\vdots & \vdots & \vdots & \vdots & \vdots & \ddots & \vdots\\
\mathsf{0} & \mathsf{0} & \mathsf{0} & \ldots & \mathsf{0} & \ldots &
\fbox{$\sigma_{j}$\textsf{\ \ \ }$^{\left(  \mathbf{n-2}\right)  }$}\\
\fbox{$\sigma_{j}$\textsf{\ \ \ }$^{\left(  \mathbf{n-1}\right)  }$} &
\mathsf{0} & \mathsf{0} & \ldots & \mathsf{0} & \ldots & \mathsf{0}%
\end{array}
\right)  ,\ \ \ \ \ j=0,1,2,3. \label{sj}%
\end{equation}

In this notation $\mathit{\Sigma}_{0}=\mathbf{E}$ from (\ref{ee}).

\begin{proposition}
The full $\mathit{\Sigma}$-matrices are in the polyadic group $SU^{\left[
n\right]  }\left(  2\right)  $.
\end{proposition}

\begin{proof}
This follows directly from (\ref{mx}) with all $x_{j}^{\left(  \mathbf{k}%
\right)  }=1$, $j=0,1,2,3$, $k=1,\ldots,n-1$.
\end{proof}

Using the full $\mathit{\Sigma}$-matrices we could derive the restricted
version of (\ref{mxs}) as (cf. the binary case (\ref{ms}))%
\begin{align}
&  \overline{\mathbf{M}}=x_{0}\mathit{\Sigma}_{0}+i\overrightarrow
{x}\overrightarrow{\mathit{\Sigma}},\ \ \ \ \overrightarrow{\mathit{\Sigma}%
}=\left(  \mathit{\Sigma}_{1},\mathit{\Sigma}_{2},\mathit{\Sigma}_{3}\right)
,\label{mb}\\
&  x_{j}=x_{j}^{\left(  \mathbf{1}\right)  }=x_{j}^{\left(  \mathbf{2}\right)
}=\ldots=x_{j}^{\left(  \mathbf{n-1}\right)  },\ \ j=0,1,2,3,\label{xx}\\
&  x_{0}^{2}+\overrightarrow{x}^{2}=1. \label{xx1}%
\end{align}

\begin{proposition}
The polyadic matrices (\ref{mb}) form a subgroup $\mathbf{SU}^{\left[
n\right]  }\left(  2\right)  $ of $SU^{\left[  n\right]  }\left(  2\right)  $.
\end{proposition}

\begin{proof}
The restricted matrices $\overline{\mathbf{M}}$ (\ref{mb}) correspond to the
cyclic shift block matrices (\ref{mn}) having equal entries. From the
component multiplication rules (\ref{mm1})--(\ref{mmn}), where only one
equation remains now, it follows that the set $\left\{  \overline{\mathbf{M}%
}\right\}  $ is a subgroup of $SU^{\left[  n\right]  }\left(  2\right)  $.
\end{proof}

In terms of the full $\mathit{\Sigma}$-matrices the $n$-ary multiplication in
the subgroup $\mathbf{SU}^{\left[  n\right]  }\left(  2\right)  $ is given by
the single relation (instead of (\ref{mm1})--(\ref{mmn}))%
\begin{equation}
\overset{n}{\overbrace{\left(  x_{0}^{\prime}\mathit{\Sigma}_{0}%
+i\overrightarrow{x^{\prime}}\overrightarrow{\mathit{\Sigma}}\right)  \left(
x_{0}^{\prime\prime}\mathit{\Sigma}_{0}+i\overrightarrow{x^{\prime\prime}%
}\overrightarrow{\mathit{\Sigma}}\right)  \ldots\left(  x_{0}^{\prime
\prime\prime}\mathit{\Sigma}_{0}+i\overrightarrow{x^{\prime\prime\prime}%
}\overrightarrow{\mathit{\Sigma}}\right)  }}=x_{0}\mathit{\Sigma}%
_{0}+i\overrightarrow{x}\overrightarrow{\mathit{\Sigma}}. \label{xs0}%
\end{equation}

To obtain the multiplication in the subgroup, the Cayley table for the full
$\mathit{\Sigma}$-matrices is needed. Some obvious $n$-ary products (products
of other lower arities are not closed) are the following%
\begin{align}
\left(  \mathit{\Sigma}_{0}\right)  ^{n}  &  =\mathit{\Sigma}_{0}%
,\ \ \ \ \ \ \left(  \mathit{\Sigma}_{0}\right)  ^{n-1}\overrightarrow
{\mathit{\Sigma}}=\overrightarrow{\mathit{\Sigma}},\label{sn1}\\
\left(  \mathit{\Sigma}_{j}\right)  ^{n}  &  =\left\{
\begin{array}
[c]{c}%
\mathit{\Sigma}_{0},\ \ \ \text{if }n\text{ is even,}\\
\mathit{\Sigma}_{j},\ \ \ \text{if }n\text{ is odd.}%
\end{array}
\right.  ,\ \ j=1,2,3. \label{sn2}%
\end{align}
Therefore we have

\begin{assertion}
For odd arity all full $\mathit{\Sigma}$-matrices are $n$-ary idempotents,
while for even arity $\mathit{\Sigma}_{j}$ are $n$-ary reflections, because
$\mathit{\Sigma}_{0}\equiv\mathbf{E}$.
\end{assertion}

Then to find the $n$-ary multiplication (\ref{xs0}) in parameters $x_{j}$, we
need the expressions of $\sigma$-matrices up to $n$, because of the exact form
(\ref{sj}) of full $\mathit{\Sigma}$-matrices.

\begin{example}
For the case $n=4$ we have quaternary full $\mathit{\Sigma}$-matrices%
\begin{equation}
\mathit{\Sigma}_{0}=\left(
\begin{array}
[c]{ccc}%
\mathsf{0} & I_{2} & \mathsf{0}\\
\mathsf{0} & \mathsf{0} & I_{2}\\
I_{2} & \mathsf{0} & \mathsf{0}%
\end{array}
\right)  ,\ \ \ \ \ \ \ \mathit{\Sigma}_{j}=\left(
\begin{array}
[c]{ccc}%
\mathsf{0} & \sigma_{j} & \mathsf{0}\\
\mathsf{0} & \mathsf{0} & \sigma_{j}\\
\sigma_{j} & \mathsf{0} & \mathsf{0}%
\end{array}
\right)  ,\ \ j=1,2,3, \label{s4}%
\end{equation}
such that only the product of $4$ of them is closed (and of polyadic power
$\ell$, having $3\ell+1$ multipliers), and $\mathit{\Sigma}_{j}$ are $4$-ary
reflections, since $\left(  \mathit{\Sigma}_{j}\right)  ^{4}=\mathbf{E}%
^{\left[  \mathbf{4}\right]  }=\mathbf{E}$ (see (\ref{ee})).
\end{example}

We now try to construct the full polyadic group $SU^{\left[  n\right]
}\left(  2\right)  $ using the Hadamard product ($\odot$) (or element-wise
product, Schur product). Introduce the $2\left(  n-1\right)  \times2\left(
n-1\right)  $ cyclic shift block matrix of parameters of $SU^{\left[
n\right]  }\left(  2\right)  $ (see (\ref{mk}))%
\begin{align}
\mathbf{X}_{j}=  &  \left(
\begin{array}
[c]{ccccccc}%
\mathsf{0} & x_{j}^{\left(  \mathbf{1}\right)  }I_{2} & \mathsf{0} & \ldots &
\mathsf{0} & \ldots & \mathsf{0}\\
\mathsf{0} & \mathsf{0} & x_{j}^{\left(  \mathbf{2}\right)  }I_{2} & \ldots &
\mathsf{0} & \ldots & \mathsf{0}\\
\mathsf{0} & \mathsf{0} & \mathsf{0} & \ddots & \vdots & \ldots & \vdots\\
\mathsf{0} & \mathsf{0} & \mathsf{0} & \ldots & x_{j}^{\left(  \mathbf{k}%
\right)  }I_{2} & \ldots & \mathsf{0}\\
\vdots & \vdots & \vdots & \vdots & \vdots & \ddots & \vdots\\
\mathsf{0} & \mathsf{0} & \mathsf{0} & \ldots & \mathsf{0} & \ldots &
x_{j}^{\left(  \mathbf{n-2}\right)  }I_{2}\\
x_{j}^{\left(  \mathbf{n-1}\right)  }I_{2} & \mathsf{0} & \mathsf{0} & \ldots
& \mathsf{0} & \ldots & \mathsf{0}%
\end{array}
\right)  ,\ \ j=0,1,2,3.\label{xxi}\\[10pt]
&  \left(  x_{0}^{\left(  \mathbf{k}\right)  }\right)  ^{2}+\left(
\overrightarrow{x}^{\left(  \mathbf{k}\right)  }\right)  ^{2}=1,\ \ k=1,\ldots
n-1.
\end{align}

Then, instead of (\ref{mxs}) and (\ref{mb}) we have

\begin{theorem}
The full polyadic group $SU^{\left[  n\right]  }\left(  2\right)  $ matrix
(\ref{mx}) can be represented as an expansion in the polyadic full
$\mathit{\Sigma}$-matrices by%
\begin{equation}
\mathbf{M}=\mathbf{X}_{0}\odot\mathit{\Sigma}_{0}+i\mathbf{X}_{1}%
\odot\mathit{\Sigma}_{1}+i\mathbf{X}_{2}\odot\mathit{\Sigma}_{2}%
+i\mathbf{X}_{3}\odot\mathit{\Sigma}_{3}, \label{mx0}%
\end{equation}
where $\left(  \odot\right)  $ is the Hadamard product.
\end{theorem}

\begin{proof}
This follows from (\ref{mx}), (\ref{sj}) and (\ref{xxi}), in both directions.
\end{proof}

Note that (\ref{mx0}) can be treated as a polyadic analog of the
\textquotedblleft multiplication by scalars\textquotedblright\ (some kind of
module) which is consistent with the cyclic shift block matrix structure.

\section{\textsc{Ternary }$SU(2)$\textsc{ and }$\mathit{\Sigma}$%
\textsc{-matrices}}

Let us consider the smallest arity case $n=3$. In manifest form an element of
the $6$-parameter matrix ternary group $SU^{\left[  \mathbf{3}\right]
}\left(  2\right)  $ is (cf. (\ref{m}))%
\begin{align}
\mathbf{M}=  &  \left(
\begin{array}
[c]{cccc}%
0 & 0 & x_{0}^{\left(  \mathbf{1}\right)  }+ix_{3}^{\left(  \mathbf{1}\right)
} & ix_{1}^{\left(  \mathbf{1}\right)  }+x_{2}^{\left(  \mathbf{1}\right)  }\\
0 & 0 & ix_{1}^{\left(  \mathbf{1}\right)  }-x_{2}^{\left(  \mathbf{1}\right)
} & x_{0}^{\left(  \mathbf{1}\right)  }-ix_{3}^{\left(  \mathbf{1}\right)  }\\
x_{0}^{\left(  \mathbf{2}\right)  }+ix_{3}^{\left(  \mathbf{2}\right)  } &
ix_{1}^{\left(  \mathbf{2}\right)  }+x_{2}^{\left(  \mathbf{2}\right)  } & 0 &
0\\
ix_{1}^{\left(  \mathbf{2}\right)  }-x_{2}^{\left(  \mathbf{2}\right)  } &
x_{0}^{\left(  \mathbf{2}\right)  }-ix_{3}^{\left(  \mathbf{2}\right)  } & 0 &
0
\end{array}
\right)  ,\label{m3}\\
&  \left(  x_{0}^{\left(  \mathbf{1}\right)  }\right)  ^{2}+\left(
\overrightarrow{x}^{\left(  \mathbf{1}\right)  }\right)  ^{2}=\left(
x_{0}^{\left(  \mathbf{2}\right)  }\right)  ^{2}+\left(  \overrightarrow
{x}^{\left(  \mathbf{2}\right)  }\right)  ^{2}=1,\ \ \ x_{j}^{\left(
\mathbf{1}\right)  },x_{j}^{\left(  \mathbf{2}\right)  }\in\mathbb{R}%
,\ \ \ j=0,1,2,3.
\end{align}

The direct multiplication table of $SU^{\left[  \mathbf{3}\right]  }\left(
2\right)  $ in terms of parameters (as (\ref{a})--\ref{d})) is too cumbersome
to display here.

There are eight $4\times4$ ternary elementary $\Sigma$-matrices%
\begin{equation}
\Sigma_{j}^{\left(  \mathbf{1}\right)  }=\left(
\begin{array}
[c]{cc}%
\mathsf{0} & \sigma_{j}\\
\mathsf{0} & \mathsf{0}%
\end{array}
\right)  ,\ \ \ \Sigma_{j}^{\left(  \mathbf{2}\right)  }=\left(
\begin{array}
[c]{cc}%
\mathsf{0} & \mathsf{0}\\
\sigma_{j} & \mathsf{0}%
\end{array}
\right)  ,\ \ \ \ \ \ \ j=0,1,2,3\label{s12}%
\end{equation}
which are ternary nilpotent%
\begin{equation}
\left(  \Sigma_{j}^{\left(  \mathbf{1}\right)  }\right)  ^{3}=\mathbf{0}%
,\ \ \ \ \ \ \ \ \left(  \Sigma_{j}^{\left(  \mathbf{2}\right)  }\right)
^{3}=\mathbf{0}.\label{ss0}%
\end{equation}
We recall the products of two and three $\sigma$-matrices%
\begin{align}
\sigma_{k}\sigma_{l} &  =\delta_{kl}\sigma_{0}+i\epsilon_{klm}\sigma
_{m},\label{s1}\\
\sigma_{k}\sigma_{l}\sigma_{m} &  =\delta_{kl}\sigma_{m}-\delta_{km}\sigma
_{l}+\delta_{lm}\sigma_{k}+i\epsilon_{klm}\sigma_{0}.\label{s2}%
\end{align}

\begin{proposition}
The Cayley table of ternary elementary $\Sigma$-matrices (\ref{s12}) can be
obtained from their multiplication%
\begin{align}
\mu^{\left[  \mathbf{3}\right]  }\left[  \Sigma_{k}^{\left(  \mathbf{1}%
\right)  },\Sigma_{l}^{\left(  \mathbf{2}\right)  },\Sigma_{m}^{\left(
\mathbf{1}\right)  }\right]   &  =\Sigma_{k}^{\left(  \mathbf{1}\right)
}\Sigma_{l}^{\left(  \mathbf{2}\right)  }\Sigma_{m}^{\left(  \mathbf{1}%
\right)  }=\delta_{kl}\Sigma_{m}^{\left(  \mathbf{1}\right)  }-\delta
_{km}\Sigma_{l}^{\left(  \mathbf{1}\right)  }+\delta_{lm}\Sigma_{k}^{\left(
\mathbf{1}\right)  }+i\epsilon_{klm}\Sigma_{0}^{\left(  \mathbf{1}\right)
},\label{ct1}\\
\mu^{\left[  \mathbf{3}\right]  }\left[  \Sigma_{k}^{\left(  \mathbf{2}%
\right)  },\Sigma_{l}^{\left(  \mathbf{1}\right)  },\Sigma_{m}^{\left(
\mathbf{2}\right)  }\right]   &  =\Sigma_{k}^{\left(  \mathbf{2}\right)
}\Sigma_{l}^{\left(  \mathbf{1}\right)  }\Sigma_{m}^{\left(  \mathbf{2}%
\right)  }=\delta_{kl}\Sigma_{m}^{\left(  \mathbf{2}\right)  }-\delta
_{km}\Sigma_{l}^{\left(  \mathbf{2}\right)  }+\delta_{lm}\Sigma_{k}^{\left(
\mathbf{2}\right)  }+i\epsilon_{klm}\Sigma_{0}^{\left(  \mathbf{2}\right)
},\label{ct2}%
\end{align}
while other products give zero $4\times4$ matrix $\mathbf{0}$.
\end{proposition}

\begin{proof}
We have in components%
\begin{align}
\Sigma_{k}^{\left(  \mathbf{1}\right)  }\Sigma_{l}^{\left(  \mathbf{2}\right)
}\Sigma_{m}^{\left(  \mathbf{1}\right)  }  &  =\left(
\begin{array}
[c]{cc}%
\mathsf{0} & \sigma_{k}\sigma_{l}\sigma_{m}\\
\mathsf{0} & \mathsf{0}%
\end{array}
\right)  ,\\
\Sigma_{k}^{\left(  \mathbf{2}\right)  }\Sigma_{l}^{\left(  \mathbf{1}\right)
}\Sigma_{m}^{\left(  \mathbf{2}\right)  }  &  =\left(
\begin{array}
[c]{cc}%
\mathsf{0} & \mathsf{0}\\
\sigma_{k}\sigma_{l}\sigma_{m} & \mathsf{0}%
\end{array}
\right)  ,
\end{align}
and then use (\ref{s2}).
\end{proof}

We introduce the ternary commutator and anti-commutator by analogy with
(\ref{c2})-(\ref{c2a})%
\begin{align}
\left[  a,b,c\right]  ^{\left[  3\right]  }  &
=abc+bca+cab-acb-bac-cba,\label{tc}\\
\left\{  a,b,c\right\}  ^{\left[  3\right]  }  &  =abc+bca+cab+acb+bac+cba,
\label{ta}%
\end{align}
where $a,b,c$ are $\Sigma$-matrices. Using (\ref{ct1})--(\ref{ct2}), for
elementary $\Sigma$-matrices we find that ternary commutators are proportional
to the corresponding ternary units (cf. (\ref{c2})--(\ref{c2a}))%
\begin{align}
\left[  \Sigma_{k}^{\left(  \mathbf{1}\right)  }\Sigma_{l}^{\left(
\mathbf{1}\right)  }\Sigma_{m}^{\left(  \mathbf{2}\right)  }\right]  ^{\left[
3\right]  }  &  =\left[  \Sigma_{k}^{\left(  \mathbf{1}\right)  }\Sigma
_{l}^{\left(  \mathbf{2}\right)  }\Sigma_{m}^{\left(  \mathbf{1}\right)
}\right]  ^{\left[  3\right]  }=\left[  \Sigma_{k}^{\left(  \mathbf{2}\right)
}\Sigma_{l}^{\left(  \mathbf{1}\right)  }\Sigma_{m}^{\left(  \mathbf{1}%
\right)  }\right]  ^{\left[  3\right]  }=2i\epsilon_{klm}\Sigma_{0}^{\left(
\mathbf{1}\right)  },\label{c3}\\
\left[  \Sigma_{k}^{\left(  \mathbf{2}\right)  }\Sigma_{l}^{\left(
\mathbf{2}\right)  }\Sigma_{m}^{\left(  \mathbf{1}\right)  }\right]  ^{\left[
3\right]  }  &  =\left[  \Sigma_{k}^{\left(  \mathbf{2}\right)  }\Sigma
_{l}^{\left(  \mathbf{1}\right)  }\Sigma_{m}^{\left(  \mathbf{2}\right)
}\right]  ^{\left[  3\right]  }=\left[  \Sigma_{k}^{\left(  \mathbf{1}\right)
}\Sigma_{l}^{\left(  \mathbf{2}\right)  }\Sigma_{m}^{\left(  \mathbf{2}%
\right)  }\right]  ^{\left[  3\right]  }=2i\epsilon_{klm}\Sigma_{0}^{\left(
\mathbf{2}\right)  },
\end{align}
and the ternary anti-commutators are%
\begin{align}
\left\{  \Sigma_{k}^{\left(  \mathbf{1}\right)  }\Sigma_{l}^{\left(
\mathbf{1}\right)  }\Sigma_{m}^{\left(  \mathbf{2}\right)  }\right\}
^{\left[  3\right]  }  &  =-2\delta_{kl}\Sigma_{m}^{\left(  \mathbf{1}\right)
}+2\delta_{km}\Sigma_{l}^{\left(  \mathbf{1}\right)  }+2\delta_{lm}\Sigma
_{k}^{\left(  \mathbf{1}\right)  },\\
\left\{  \Sigma_{k}^{\left(  \mathbf{1}\right)  }\Sigma_{l}^{\left(
\mathbf{2}\right)  }\Sigma_{m}^{\left(  \mathbf{1}\right)  }\right\}
^{\left[  3\right]  }  &  =2\delta_{kl}\Sigma_{m}^{\left(  \mathbf{1}\right)
}-2\delta_{km}\Sigma_{l}^{\left(  \mathbf{1}\right)  }+2\delta_{lm}\Sigma
_{k}^{\left(  \mathbf{1}\right)  },\\
\left\{  \Sigma_{k}^{\left(  \mathbf{2}\right)  }\Sigma_{l}^{\left(
\mathbf{1}\right)  }\Sigma_{m}^{\left(  \mathbf{1}\right)  }\right\}
^{\left[  3\right]  }  &  =2\delta_{kl}\Sigma_{m}^{\left(  \mathbf{1}\right)
}+2\delta_{km}\Sigma_{l}^{\left(  \mathbf{1}\right)  }-2\delta_{lm}\Sigma
_{k}^{\left(  \mathbf{1}\right)  },
\end{align}%
\begin{align}
\left\{  \Sigma_{k}^{\left(  \mathbf{2}\right)  }\Sigma_{l}^{\left(
\mathbf{2}\right)  }\Sigma_{m}^{\left(  \mathbf{1}\right)  }\right\}
^{\left[  3\right]  }  &  =-2\delta_{kl}\Sigma_{m}^{\left(  \mathbf{2}\right)
}+2\delta_{km}\Sigma_{l}^{\left(  \mathbf{2}\right)  }+2\delta_{lm}\Sigma
_{k}^{\left(  \mathbf{2}\right)  },\\
\left\{  \Sigma_{k}^{\left(  \mathbf{2}\right)  }\Sigma_{l}^{\left(
\mathbf{1}\right)  }\Sigma_{m}^{\left(  \mathbf{2}\right)  }\right\}
^{\left[  3\right]  }  &  =2\delta_{kl}\Sigma_{m}^{\left(  \mathbf{2}\right)
}-2\delta_{km}\Sigma_{l}^{\left(  \mathbf{2}\right)  }+2\delta_{lm}\Sigma
_{k}^{\left(  \mathbf{2}\right)  },\\
\left\{  \Sigma_{k}^{\left(  \mathbf{1}\right)  }\Sigma_{l}^{\left(
\mathbf{2}\right)  }\Sigma_{m}^{\left(  \mathbf{2}\right)  }\right\}
^{\left[  3\right]  }  &  =2\delta_{kl}\Sigma_{m}^{\left(  \mathbf{2}\right)
}+2\delta_{km}\Sigma_{l}^{\left(  \mathbf{2}\right)  }-2\delta_{lm}\Sigma
_{k}^{\left(  \mathbf{2}\right)  }, \label{sk2}%
\end{align}
while other (anti)commutators vanish.

The expansion of an $SU^{\left[  \mathbf{3}\right]  }\left(  2\right)  $
matrix (\ref{mxs}) in terms of the elementary $\Sigma$-matrices becomes%
\begin{align}
&  \mathbf{M}=x_{0}^{\left(  \mathbf{1}\right)  }\Sigma_{0}^{\left(
\mathbf{1}\right)  }+i\overrightarrow{x}^{\left(  \mathbf{1}\right)
}\overrightarrow{\Sigma}^{\left(  \mathbf{1}\right)  }+x_{0}^{\left(
\mathbf{2}\right)  }\Sigma_{0}^{\left(  \mathbf{2}\right)  }+i\overrightarrow
{x}^{\left(  \mathbf{2}\right)  }\overrightarrow{\Sigma}^{\left(
\mathbf{2}\right)  },\\
&  \overrightarrow{\Sigma}^{\left(  \mathbf{1,2}\right)  }=\left(  \Sigma
_{1}^{\left(  \mathbf{1,2}\right)  },\Sigma_{2}^{\left(  \mathbf{1,2}\right)
},\Sigma_{3}^{\left(  \mathbf{1,2}\right)  }\right)  .
\end{align}

The presentation of $\mathbf{M\in}SU^{\left[  \mathbf{3}\right]  }\left(
2\right)  $ in terms of the Hadamard product is given by (\ref{mx0}), where
now the block-matrices of parameters (\ref{xxi}) are%
\begin{equation}
\mathbf{X}_{j}=\left(
\begin{array}
[c]{cc}%
\mathsf{0} & x_{j}^{\left(  \mathbf{1}\right)  }I_{2}\\
x_{j}^{\left(  \mathbf{2}\right)  }I_{2} & \mathsf{0}%
\end{array}
\right)  ,\ \ \ j=0,1,2,3, \label{xi}%
\end{equation}
and the full ternary $\mathit{\Sigma}$-matrices are%
\begin{equation}
\mathit{\Sigma}_{j}=\Sigma_{j}^{\left(  \mathbf{1}\right)  }+\Sigma
_{j}^{\left(  \mathbf{2}\right)  }=\left(
\begin{array}
[c]{cc}%
\mathsf{0} & \sigma_{j}\\
\sigma_{j} & \mathsf{0}%
\end{array}
\right)  ,\ \ \ j=0,1,2,3. \label{sss}%
\end{equation}

\begin{proposition}
The Cayley table of the full ternary $\mathit{\Sigma}$-matrices is given by%
\begin{align}
\mathit{\Sigma}_{k}\mathit{\Sigma}_{l}\mathit{\Sigma}_{m}  &  =\delta
_{kl}\mathit{\Sigma}_{m}-\delta_{km}\mathit{\Sigma}_{l}+\delta_{lm}%
\mathit{\Sigma}_{k}+i\epsilon_{klm}\mathit{\Sigma}_{0},\label{skl}\\[1pt]
\mathit{\Sigma}_{k}\mathit{\Sigma}_{l}\mathit{\Sigma}_{0}  &  =\mathit{\Sigma
}_{k}\mathit{\Sigma}_{0}\mathit{\Sigma}_{l}=\mathit{\Sigma}_{0}\mathit{\Sigma
}_{k}\mathit{\Sigma}_{l}=\delta_{kl}\mathit{\Sigma}_{0}+i\epsilon
_{klm}\mathit{\Sigma}_{m},\label{skl2}\\
\mathit{\Sigma}_{k}\mathit{\Sigma}_{0}\mathit{\Sigma}_{0}  &  =\mathit{\Sigma
}_{0}\mathit{\Sigma}_{k}\mathit{\Sigma}_{0}=\mathit{\Sigma}_{0}\mathit{\Sigma
}_{0}\mathit{\Sigma}_{k}=\mathit{\Sigma}_{k},\ \ \ k,l,m=1,2,3. \label{skl3}%
\end{align}

\end{proposition}

\begin{proof}
This follows from (\ref{ct1})-(\ref{ct2}) and (\ref{sss}), taking into account
the ternary nilpotency (\ref{ss0}) of the elementary $\Sigma$-matrices.
\end{proof}

It follows from (\ref{sn1}) and (\ref{skl}) that all full $\mathit{\Sigma}%
$-matrices are ternary idempotents (cf. (\ref{sn2}) and (\ref{ss0}))%
\begin{equation}
\left(  \mathit{\Sigma}_{j}\right)  ^{3}=\mathit{\Sigma}_{j}%
,\ \ \ \ \ \ \ \ j=0,1,2,3.
\end{equation}

In the same way, from (\ref{c3})--(\ref{sk2}) we get for full ternary
$\mathit{\Sigma}$-matrices the ternary commutators (cf. (\ref{c2}%
)--(\ref{c2a}))%
\begin{equation}
\left[  \mathit{\Sigma}_{k}\mathit{\Sigma}_{l}\mathit{\Sigma}_{m}\right]
^{\left[  3\right]  }=6i\epsilon_{klm}\Sigma_{0},\ \ \ k,l,m=1,2,3,
\end{equation}
and ternary anti-commutators%
\begin{equation}
\left\{  \mathit{\Sigma}_{k}\mathit{\Sigma}_{l}\mathit{\Sigma}_{m}\right\}
^{\left[  3\right]  }=2\delta_{kl}\mathit{\Sigma}_{m}+2\delta_{km}%
\mathit{\Sigma}_{l}+2\delta_{lm}\mathit{\Sigma}_{k}. \label{sks}%
\end{equation}

Using (\ref{mx0}), (\ref{xi}) and (\ref{sss}), we obtain the $SU^{\left[
\mathbf{3}\right]  }\left(  2\right)  $ matrix (cf. (\ref{mx}))%
\begin{equation}
\mathbf{M}=\left(
\begin{array}
[c]{cc}%
\mathsf{0} & x_{0}^{\left(  \mathbf{1}\right)  }\sigma_{0}+i\overrightarrow
{x}^{\left(  \mathbf{1}\right)  }\overrightarrow{\sigma}\\
x_{0}^{\left(  \mathbf{2}\right)  }\sigma_{0}+i\overrightarrow{x}^{\left(
\mathbf{2}\right)  }\overrightarrow{\sigma} & \mathsf{0}%
\end{array}
\right)  .
\end{equation}

The ternary multiplication in $SU^{\left[  \mathbf{3}\right]  }\left(
2\right)  $ can be derived from the general case (\ref{mm1})--(\ref{mmn}), as
follows%
\begin{align}
x_{0}^{\left(  \mathbf{1}\right)  }\sigma_{0}+i\overrightarrow{x}^{\left(
\mathbf{1}\right)  }\overrightarrow{\sigma}  &  =\left(  x_{0}^{\left(
\mathbf{1}\right)  \prime}\sigma_{0}+i\overrightarrow{x}^{\left(
\mathbf{1}\right)  \prime}\overrightarrow{\sigma}\right)  \left(
x_{0}^{\left(  \mathbf{2}\right)  \prime\prime}\sigma_{0}+i\overrightarrow
{x}^{\left(  \mathbf{2}\right)  \prime\prime}\overrightarrow{\sigma}\right)
\left(  x_{0}^{\left(  \mathbf{1}\right)  \prime\prime\prime}\sigma
_{0}+i\overrightarrow{x}^{\left(  \mathbf{1}\right)  \prime\prime\prime
}\overrightarrow{\sigma}\right)  ,\label{x11}\\
x_{0}^{\left(  \mathbf{2}\right)  }\sigma_{0}+i\overrightarrow{x}^{\left(
\mathbf{2}\right)  }\overrightarrow{\sigma}  &  =\left(  x_{0}^{\left(
\mathbf{2}\right)  \prime}\sigma_{0}+i\overrightarrow{x}^{\left(
\mathbf{2}\right)  \prime}\overrightarrow{\sigma}\right)  \left(
x_{0}^{\left(  \mathbf{1}\right)  \prime\prime}\sigma_{0}+i\overrightarrow
{x}^{\left(  \mathbf{1}\right)  \prime\prime}\overrightarrow{\sigma}\right)
\left(  x_{0}^{\left(  \mathbf{2}\right)  \prime\prime\prime}\sigma
_{0}+i\overrightarrow{x}^{\left(  \mathbf{2}\right)  \prime\prime\prime
}\overrightarrow{\sigma}\right)  . \label{x12}%
\end{align}

Using (\ref{s1})--(\ref{s2}), we get the exact form of the ternary group
$SU^{\left[  \mathbf{3}\right]  }\left(  2\right)  $ multiplication
(\ref{x11})--(\ref{x12}) in terms of the group parameters as (cf.
(\ref{a})--(\ref{d}))%
\begin{align}
x_{0}^{\left(  \mathbf{1}\right)  }  &  =x_{0}^{\left(  \mathbf{1}\right)
\prime}x_{0}^{\left(  \mathbf{2}\right)  \prime\prime}x_{0}^{\left(
\mathbf{1}\right)  \prime\prime\prime}-x_{0}^{\left(  \mathbf{1}\right)
\prime}\left(  \overrightarrow{x}^{\left(  \mathbf{2}\right)  \prime\prime
}\overrightarrow{x}^{\left(  \mathbf{1}\right)  \prime\prime\prime}\right)
-x_{0}^{\left(  \mathbf{2}\right)  \prime\prime}\left(  \overrightarrow
{x}^{\left(  \mathbf{1}\right)  \prime}\overrightarrow{x}^{\left(
\mathbf{1}\right)  \prime\prime\prime}\right) \nonumber\\
&  -x_{0}^{\left(  \mathbf{1}\right)  \prime\prime\prime}\left(
\overrightarrow{x}^{\left(  \mathbf{1}\right)  \prime}\overrightarrow
{x}^{\left(  \mathbf{2}\right)  \prime\prime}\right)  +\overrightarrow
{x}^{\left(  \mathbf{1}\right)  \prime}\left(  \overrightarrow{x}^{\left(
\mathbf{2}\right)  \prime\prime}\times\overrightarrow{x}^{\left(
\mathbf{1}\right)  \prime\prime\prime}\right)  ,\label{x10}\\
\overrightarrow{x}^{\left(  \mathbf{1}\right)  }  &  =\overrightarrow
{x}^{\left(  \mathbf{2}\right)  \prime\prime}\left(  \overrightarrow
{x}^{\left(  \mathbf{1}\right)  \prime}\times\overrightarrow{x}^{\left(
\mathbf{1}\right)  \prime\prime\prime}\right)  -\overrightarrow{x}^{\left(
\mathbf{1}\right)  \prime}\left(  \overrightarrow{x}^{\left(  \mathbf{2}%
\right)  \prime\prime}\times\overrightarrow{x}^{\left(  \mathbf{1}\right)
\prime\prime\prime}\right)  -\overrightarrow{x}^{\left(  \mathbf{1}\right)
\prime\prime\prime}\left(  \overrightarrow{x}^{\left(  \mathbf{2}\right)
\prime\prime}\times\overrightarrow{x}^{\left(  \mathbf{1}\right)  \prime
}\right) \nonumber\\
&  -x_{0}^{\left(  \mathbf{1}\right)  \prime}\left(  \overrightarrow
{x}^{\left(  \mathbf{2}\right)  \prime\prime}\times\overrightarrow{x}^{\left(
\mathbf{1}\right)  \prime\prime\prime}\right)  -x_{0}^{\left(  \mathbf{2}%
\right)  \prime\prime}\left(  \overrightarrow{x}^{\left(  \mathbf{1}\right)
\prime}\times\overrightarrow{x}^{\left(  \mathbf{1}\right)  \prime\prime
\prime}\right)  -x_{0}^{\left(  \mathbf{1}\right)  \prime\prime\prime}\left(
\overrightarrow{x}^{\left(  \mathbf{1}\right)  \prime}\times\overrightarrow
{x}^{\left(  \mathbf{2}\right)  \prime\prime}\right)  , \label{x111}%
\end{align}%
\begin{align}
x_{0}^{\left(  \mathbf{2}\right)  }  &  =x_{0}^{\left(  \mathbf{2}\right)
\prime}x_{0}^{\left(  \mathbf{1}\right)  \prime\prime}x_{0}^{\left(
\mathbf{2}\right)  \prime\prime\prime}-x_{0}^{\left(  \mathbf{2}\right)
\prime}\left(  \overrightarrow{x}^{\left(  \mathbf{1}\right)  \prime\prime
}\overrightarrow{x}^{\left(  \mathbf{2}\right)  \prime\prime\prime}\right)
-x_{0}^{\left(  \mathbf{1}\right)  \prime\prime}\left(  \overrightarrow
{x}^{\left(  \mathbf{2}\right)  \prime}\overrightarrow{x}^{\left(
\mathbf{2}\right)  \prime\prime\prime}\right) \nonumber\\
&  -x_{0}^{\left(  \mathbf{2}\right)  \prime\prime\prime}\left(
\overrightarrow{x}^{\left(  \mathbf{2}\right)  \prime}\overrightarrow
{x}^{\left(  \mathbf{1}\right)  \prime\prime}\right)  +\overrightarrow
{x}^{\left(  \mathbf{2}\right)  \prime}\left(  \overrightarrow{x}^{\left(
\mathbf{1}\right)  \prime\prime}\times\overrightarrow{x}^{\left(
\mathbf{2}\right)  \prime\prime\prime}\right)  ,\label{x20}\\
\overrightarrow{x}^{\left(  \mathbf{2}\right)  }  &  =\overrightarrow
{x}^{\left(  \mathbf{1}\right)  \prime\prime}\left(  \overrightarrow
{x}^{\left(  \mathbf{2}\right)  \prime}\times\overrightarrow{x}^{\left(
\mathbf{2}\right)  \prime\prime\prime}\right)  -\overrightarrow{x}^{\left(
\mathbf{2}\right)  \prime}\left(  \overrightarrow{x}^{\left(  \mathbf{1}%
\right)  \prime\prime}\times\overrightarrow{x}^{\left(  \mathbf{2}\right)
\prime\prime\prime}\right)  -\overrightarrow{x}^{\left(  \mathbf{2}\right)
\prime\prime\prime}\left(  \overrightarrow{x}^{\left(  \mathbf{1}\right)
\prime\prime}\times\overrightarrow{x}^{\left(  \mathbf{2}\right)  \prime
}\right) \nonumber\\
&  -x_{0}^{\left(  \mathbf{2}\right)  \prime}\left(  \overrightarrow
{x}^{\left(  \mathbf{1}\right)  \prime\prime}\times\overrightarrow{x}^{\left(
\mathbf{2}\right)  \prime\prime\prime}\right)  -x_{0}^{\left(  \mathbf{1}%
\right)  \prime\prime}\left(  \overrightarrow{x}^{\left(  \mathbf{2}\right)
\prime}\times\overrightarrow{x}^{\left(  \mathbf{2}\right)  \prime\prime
\prime}\right)  -x_{0}^{\left(  \mathbf{2}\right)  \prime\prime\prime}\left(
\overrightarrow{x}^{\left(  \mathbf{2}\right)  \prime}\times\overrightarrow
{x}^{\left(  \mathbf{1}\right)  \prime\prime}\right)  , \label{x21}%
\end{align}
where $\overrightarrow{x}\times\overrightarrow{y}$ is the standard $3D$ vector product.

The additional one parameter ternary idempotents in $SU^{\left[
\mathbf{3}\right]  }\left(  2\right)  $ are real $4\times4$ matrices (cf.
(\ref{el1}))%
\begin{equation}
\mathbf{E}_{\operatorname*{id}}\left(  a\right)  =\mathbf{E}^{\left[
\mathbf{3}\right]  }\left(  a\right)  =\left(
\begin{array}
[c]{cc}%
\mathsf{0} & aI_{2}\\
\frac{1}{a}I_{2} & \mathsf{0}%
\end{array}
\right)  ,\ \ \ \ a\in\mathbb{R\setminus}\left\{  0\right\}  , \label{eid}%
\end{equation}
which are also left and right identities, but not middle identities, according
to (\ref{el1})--(\ref{er2}).

The querelement of the ternary group $SU^{\left[  \mathbf{3}\right]  }\left(
2\right)  =\left\langle \left\{  \mathbf{M}\right\}  \mid\mu^{\left[
\mathbf{3}\right]  },\left(  \widetilde{\_}\right)  \right\rangle $ coincides
with the matrix inverse $\widetilde{\mathbf{M}}=\mathbf{M}^{-1}$ (in the
ternary case only), as follows from the general formula (\ref{mq}).

\section{$n$\textsc{-ary semigroups and groups of} $\Sigma$\textsc{-matrices}}

Let us briefly recall the concept of the (binary) group of $\sigma$-matrices
(or the Pauli group), which is widely used in the quantum error correction and
stabilizer codes (see, e.g. \cite{nie/chu,bal/cen/hub,djordjevic}).

\subsection{The Pauli group}

The set of four $\sigma$-matrices do not form a group, because their
multiplication is not closed: the r.h.s. of (\ref{s1}) contains the
multipliers, informally $\sigma_{k}\sigma_{l}=s\sigma_{m}$, where $\left\{
s\right\}  =\pm1$, $\pm i$, and $k,l,m=1,2,3$. One approach to obtain a closed
multiplication is by enlarging the set by new elements which arise from
consequent products. It is obvious that the enlarged set will be finite, if
the multipliers are cyclic. Indeed, in the case of $\sigma$-matrices the
multipliers themselves form the cyclic group of order $4$ generated by the
imaginary unit%
\begin{equation}
\left\{  s\right\}  =\left\langle i\right\rangle _{4}.\label{si4}%
\end{equation}

In this way, the enlarged set contains $16$ new elements $\hat{\sigma}%
_{j}\left(  r_{j}\right)  $ which are constructed from $\sigma$-matrices and
powers of $i$ as follows%
\begin{equation}
\hat{\sigma}_{j}\left(  r_{j}\right)  =i^{r_{j}}\sigma_{j}%
,\ \ \ j=0,1,2,3,\ \ \ r_{j}=0,1,2,3,\label{is}%
\end{equation}
such that $\hat{\sigma}_{j}\left(  0\right)  \equiv\sigma_{j}$ (\ref{si}). The
set of $2\times2$ invertible matrices (\ref{is}) (we call them extended
$\sigma$-matrices or sigma matrices with phase) is closed under multiplication
being associative, as matrix products, moreover it contains the identity
$I_{2}=\hat{\sigma}_{0}\left(  0\right)  $, therefore%
\begin{equation}
\hat{G}_{\sigma}=\left\langle \left\{  \hat{\sigma}_{j}\left(  r_{j}\right)
\right\}  \mid\left(  \cdot_{\sigma}\right)  ,\left(  \_\right)  ^{-1}%
,I_{2}\right\rangle \label{gs}%
\end{equation}
is a finite binary group of order $16$. We call the group $\hat{G}_{\sigma}$ a
group of phase-shifted sigma matrices, which is convenient for polyadic
extentions. It is also called the Pauli group $\mathcal{P}_{1}$, diquaternion
group $DQ_{8}$, and by structure it is isomorphic to the central product of a
cyclic group $C_{4}$ of order $4$ and the dihedral group $D_{4}$ of order 8 as
$C_{4}\circ D_{4}$ or to the semi-direct product of the quaternion group
$Q_{8}$ of order $8$ and a cyclic group $C_{2}$ of order $2$ as $Q_{8}%
\rtimes_{2}C_{2}$, as well as $D_{4}\rtimes_{2}C_{2}$, and $\left(
C_{4}\rtimes_{2}C_{2}\right)  \rtimes_{3}C_{2}$ (see, e.g.
\cite{bag/bav/rus,bes/eic/bri,con/cur/nor}).

The Cayley table of $\hat{G}_{\sigma}$ can be obtained from the manifest
multiplication of the phase-shifted $\sigma$-matrices%
\begin{equation}
\hat{\sigma}_{k}\left(  r_{k}\right)  \cdot_{\sigma}\hat{\sigma}_{l}\left(
r_{l}\right)  =\delta_{kl}\hat{\sigma}_{0}\left(  r_{k}\dotplus_{4}%
r_{l}\right)  +\left\vert \epsilon_{klm}\right\vert \hat{\sigma}_{m}\left(
r_{k}\dotplus_{4}r_{l}\dotplus_{4}1\dotplus_{4}\left(  1-\epsilon
_{klm}\right)  \right)  ,\ \ \ k,l,m=1,2,3,\label{srs}%
\end{equation}
where $r_{k}\dotplus_{4}r_{l}\equiv\left(  r_{k}+r_{l}\right)
\operatorname{mod}4$, which follows directly from (\ref{s1}) on multiplying
both sides by $i^{r_{k}}i^{r_{l}}$. In (\ref{srs}) we introduced the
\textquotedblleft phase presentation\textquotedblright\ of the Levi-Civita
symbol%
\begin{equation}
\epsilon_{klm}=\left\vert \epsilon_{klm}\right\vert e^{i\frac{\pi}%
{2}(1-\epsilon_{klm})},\ \ \ k,l,m=1,2,3,\label{ek}%
\end{equation}
which is convenient for calculation of the Cayley table for $\sigma$-matrices
with phase (\ref{is}).

\subsection{Groups of phase-shifted sigma matrices}

Let us generalize the multiplier construction of $\sigma$-matrices (\ref{is})
in the following way. Instead of the cyclic group $\left\langle i\right\rangle
_{4}$ of order $4$, we introduce higher order cyclic groups which contain
$\left\langle i\right\rangle _{4}$ as the subgroup. For that we rewrite the
phase factor in (\ref{is}) as the $4$th root of unity%
\begin{equation}
i^{r_{j}}=e^{i\frac{2\pi}{q}r_{j}},\ \ q=4,\ \ r_{j}=0,1,2,3. \label{ir}%
\end{equation}

To make the set of $q$'s finite, we propose that $q$ should be among divisors
of $360$ which are multiples of $4$ (e.g. to have an integer number of
degrees, etc.), that is%
\begin{equation}
q\in\mathsf{Q}_{12}=\left\{  \mathbf{4}%
,8,12,20,24,36,40,60,72,120,180,360\right\}  . \label{q}%
\end{equation}

In this way we obtain $12$ cyclic groups of order $q$ in which the first one
corresponds to the cyclic group of order $q=4$ (the Pauli group, in bold)
generated by the imaginary unit $\left\langle i\right\rangle _{q=4}$ (see
(\ref{si4}) and (\ref{ir})). Then, now for the fixed $q$ from (\ref{q}), the
enlarged set of $q$-generalized $\sigma$-matrices with multipliers from the
cyclic group $C_{q}$ (with $e^{i\frac{2\pi}{q}}$ as the primitive root,
instead of the imaginary unit $e^{i\frac{\pi}{2}}$) will contain $q$ different
phase factors, and therefore we have

\begin{definition}
The enlarged set of the $q$-generalized phase-shifted $\sigma$-matrices (or
multiplied by $q$-th roots of unity) contains $4q$ elements%
\begin{equation}
\hat{\sigma}_{j}^{\left(  q\right)  }\left(  r_{j}\right)  =e^{i\frac{2\pi}%
{q}r_{j}}\sigma_{j},\ \ \ j=0,1,2,3,\ \ \ r_{j}=0,1,2,\ldots q-1,\ \ \ q\in
\mathsf{Q}_{12}. \label{sq}%
\end{equation}

\end{definition}

The multiplication $\left(  \cdot_{\sigma\left(  q\right)  }\right)  $ of the
phase-shifted $\sigma$-matrices (and the corresponding Cayley table) can be
obtained from (\ref{s1}) by multiplying both sides by the factors
$e^{i\frac{2\pi}{q}r_{k}}e^{i\frac{2\pi}{q}r_{l}}$ and using the
\textquotedblleft phase presentation\textquotedblright\ of the Levi-Civita
symbol (\ref{ek}), analogous to (\ref{srs})%
\begin{align}
\hat{\sigma}_{k}^{\left(  q\right)  }\left(  r_{k}\right)  \cdot
_{\sigma\left(  q\right)  }\hat{\sigma}_{l}^{\left(  q\right)  }\left(
r_{l}\right)   &  =\delta_{kl}\hat{\sigma}_{0}^{\left(  q\right)  }\left(
r_{k}\dotplus_{q}r_{l}\right)  +\left\vert \epsilon_{klm}\right\vert
\hat{\sigma}_{m}^{\left(  q\right)  }\left(  r_{k}\dotplus_{q}r_{l}%
\dotplus_{q}\frac{q}{4}\dotplus_{q}\frac{q}{4}\left(  1-\epsilon_{klm}\right)
\right)  ,\nonumber\\
k,l,m &  =1,2,3,\ \ \ r_{k,l}=0,1,2,\ldots q-1,\ \ \ q\in\mathsf{Q}%
_{12},\label{sm}%
\end{align}
where $r_{k}\dotplus_{q}r_{l}=\left(  r_{k}+r_{l}\right)  \operatorname{mod}%
q$. The identity now is $I_{2}=\hat{\sigma}_{0}^{\left(  q\right)  }\left(
0\right)  $. Instead of an involutory (\ref{in}) we have%
\begin{equation}
\hat{\sigma}_{k}^{\left(  q\right)  }\left(  r_{k}\right)  \cdot
_{\sigma\left(  q\right)  }\hat{\sigma}_{k}^{\left(  q\right)  }\left(
r_{k}\right)  =\hat{\sigma}_{0}^{\left(  q\right)  }\left(  \left(
2r_{k}\right)  \operatorname{mod}q\right)  .\label{rm}%
\end{equation}

\begin{proposition}
The sets of the phase-shifted $\sigma$-matrices (\ref{sq}) form $12$
non-isomorphic finite groups of order $4q$ (\ref{q})%
\begin{equation}
\hat{G}_{\sigma}^{\left(  q\right)  }=\left\langle \left\{  \hat{\sigma}%
_{j}^{\left(  q\right)  }\left(  r_{j}\right)  \right\}  \mid\left(
\cdot_{\sigma\left(  q\right)  }\right)  ,\left(  \_\right)  ^{-1}%
,I_{2}\right\rangle ,\ \ \ j=0,1,2,3,\ \ \ r_{j}=0,1,2,\ldots q-1,\ \ \ q\in
\mathsf{Q}_{12}.
\end{equation}

\end{proposition}

\begin{proof}
The set is $\left\{  \hat{\sigma}_{j}^{\left(  q\right)  }\left(
r_{j}\right)  \right\}  $ closed with respect to the product (\ref{sm}),
because the factors in (\ref{sq}) form a cyclic group, it contains the
identity $I_{2}=\hat{\sigma}_{0}^{\left(  q\right)  }\left(  0\right)  $, and
every element has an inverse.
\end{proof}

Note that now the finite binary groups up to the order 2000 are listed
\cite{bes/eic/bri1}.

\subsection{The $n$-ary semigroup of elementary $\Sigma$-matrices}

The set of $4\left(  n-1\right)  $ elementary $\Sigma$-matrices $\Sigma
_{j}^{\left[  \mathbf{k}\right]  }$ (\ref{sk}) is not closed under $n$-ary
multiplication $\mu^{\left[  \mathbf{n}\right]  }$ (\ref{mnm}) by the same
reason as for the set of $4$ binary $\sigma$-matrices: the r.h.s. of their
product (see (\ref{ct1})--(\ref{ct2}) for ternary case) contains the same
multipliers $r$, where $\left\{  r\right\}  =\pm1$, $\pm i$, or $\left\{
r\right\}  =\left\langle i\right\rangle _{4}$, the cyclic group of order $4$.
So we will use now the same approach as in the binary case: by enlarging the
set of elementary $\Sigma$-matrices.

\begin{definition}
The enlarged set of the $q$-generalized phase-shifted elementary $\Sigma
$-matrices contains $4q\left(  n-1\right)  $ elements (see (\ref{is}),
(\ref{q}) and (\ref{sq}))%
\begin{equation}
\hat{\Sigma}_{j}^{\left[  \mathbf{k}\right]  \left(  q\right)  }\left(
r_{j}\right)  =e^{i\frac{2\pi}{q}r_{j}}\Sigma_{j}^{\left[  \mathbf{k}\right]
},\ \ \ j=0,1,2,3,\ \ \ r_{j}=0,1,2,\ldots q-1,\ \ \ q\in\mathsf{Q}%
_{12},\ \ \ k=1,\ldots,n-1. \label{skj}%
\end{equation}

\end{definition}

\begin{assertion}
\label{asser-0}The set $\left\{  \hat{\Sigma}_{j}^{\left[  \mathbf{k}\right]
\left(  q\right)  }\left(  r\right)  \right\}  \cup\left\{  \mathbf{0}%
\right\}  $, where $\mathbf{0}$ the $4\times4$ matrix with zero entries, is
closed with respect to $n$-ary multiplication $\mu^{\left[  \mathbf{n}\right]
}$ (\ref{mnm}).
\end{assertion}

\begin{proof}
This follows from (\ref{sk}) that non-zero $n$-ary products of $\Sigma
$-matrices are closed, as it is for the products of $n$ $q$-generalized
extended $\sigma$-matrices (\ref{sq}). Other products are equal to
$\mathbf{0}$.
\end{proof}

\begin{proposition}
The set of the $q$-generalized phase-shifted $n$-ary elementary $\Sigma
$-matrices forms $12$ non-isomorphic finite $n$-ary semigroups with zero, of
order $4q\left(  n-1\right)  +1$, $q\in\mathsf{Q}_{12}$.
\end{proposition}

\begin{proof}
It follows from \textbf{Assertion} \ref{asser-0}, and total polyadic
associativity follows from the fact that the multiplication of ordinary
matrices is associative.
\end{proof}

The formulas for $n$-ary products of $q$-generalized of $\Sigma$-matrices can
be written in concrete cases.

\begin{example}
[\textsf{Semigroup of ternary} $\Sigma$\textsf{-matrices}]The multiplication
of $q$-generalized phase-shifted ternary elementary $\Sigma$-matrices can be
obtained from (\ref{ct1})--(\ref{ct2}) by multiplying both sides by three
factors $e^{i\frac{2\pi}{q}r_{k}}e^{i\frac{2\pi}{q}r_{l}}e^{i\frac{2\pi}%
{q}r_{m}}$ and use the \textquotedblleft phase presentation\textquotedblright%
\ of the Levi-Civita symbol (\ref{ek}). In this way we get the $q$-generalized
ternary products%
\begin{align}
&  \mu_{q}^{\left[  \mathbf{3}\right]  }\left[  \hat{\Sigma}_{k}^{\left[
\mathbf{1}\right]  \left(  q\right)  }\left(  r_{k}\right)  ,\hat{\Sigma}%
_{l}^{\left[  \mathbf{2}\right]  \left(  q\right)  }\left(  r_{l}\right)
,\hat{\Sigma}_{m}^{\left[  \mathbf{1}\right]  \left(  q\right)  }\left(
r_{m}\right)  \right]  =\delta_{kl}\hat{\Sigma}_{m}^{\left[  \mathbf{1}%
\right]  \left(  q\right)  }\left(  r_{k}\dotplus_{q}r_{l}\dotplus_{q}%
r_{m}\right)  +\delta_{lm}\hat{\Sigma}_{k}^{\left[  \mathbf{1}\right]  \left(
q\right)  }\left(  r_{k}\dotplus_{q}r_{l}\dotplus_{q}r_{m}\right)  \nonumber\\
&  +\delta_{km}\hat{\Sigma}_{l}^{\left[  \mathbf{1}\right]  \left(  q\right)
}\left(  r_{k}\dotplus_{q}r_{l}\dotplus_{q}r_{m}\dotplus_{q}\frac{q}%
{2}\right)  +\left\vert \epsilon_{klm}\right\vert \hat{\Sigma}_{0}^{\left[
\mathbf{1}\right]  \left(  q\right)  }\left(  r_{k}\dotplus_{q}r_{l}%
\dotplus_{q}r_{m}\dotplus_{q}\frac{q}{4}\dotplus_{q}\frac{q}{4}\left(
1-\epsilon_{klm}\right)  \right)  ,\label{m31}%
\end{align}%
\begin{align}
&  \mu_{q}^{\left[  \mathbf{3}\right]  }\left[  \hat{\Sigma}_{k}^{\left[
\mathbf{2}\right]  \left(  q\right)  }\left(  r_{k}\right)  ,\hat{\Sigma}%
_{l}^{\left[  \mathbf{1}\right]  \left(  q\right)  }\left(  r_{l}\right)
,\hat{\Sigma}_{m}^{\left[  \mathbf{2}\right]  \left(  q\right)  }\left(
r_{m}\right)  \right]  =\delta_{kl}\hat{\Sigma}_{m}^{\left[  \mathbf{2}%
\right]  \left(  q\right)  }\left(  r_{k}\dotplus_{q}r_{l}\dotplus_{q}%
r_{m}\right)  +\delta_{lm}\hat{\Sigma}_{k}^{\left[  \mathbf{2}\right]  \left(
q\right)  }\left(  r_{k}\dotplus_{q}r_{l}\dotplus_{q}r_{m}\right)  \nonumber\\
&  +\delta_{km}\hat{\Sigma}_{l}^{\left[  \mathbf{2}\right]  \left(  q\right)
}\left(  r_{k}\dotplus_{q}r_{l}\dotplus_{q}r_{m}\dotplus_{q}\frac{q}%
{2}\right)  +\left\vert \epsilon_{klm}\right\vert \hat{\Sigma}_{0}^{\left[
\mathbf{2}\right]  \left(  q\right)  }\left(  r_{k}\dotplus_{q}r_{l}%
\dotplus_{q}r_{m}\dotplus_{q}\frac{q}{4}\dotplus_{q}\frac{q}{4}\left(
1-\epsilon_{klm}\right)  \right)  ,\label{m32}%
\end{align}
while other ones are zero $\mathbf{0}$. The ternary total associativity
follows from the ordinary associativity of the matrix product (\ref{mnm}).
Therefore, the set of $q$-generalized phase-shifted ternary elementary
$\Sigma$-matrices form a finite ternary semigroup of order $8q+1$,
$q\in\mathsf{Q}_{12}$ with the multiplication (\ref{m31})--(\ref{m32})%
\begin{equation}
\hat{S}_{\Sigma}^{\left(  q\right)  }=\left\langle \left\{  \hat{\Sigma}%
_{j}^{\left[  \mathbf{k}\right]  \left(  q\right)  }\left(  r\right)
\right\}  \cup\left\{  \mathbf{0}\right\}  \mid\mu_{q}^{\left[  \mathbf{3}%
\right]  }\right\rangle .
\end{equation}

Even in the first case $q=4$ the semigroup of ternary elementary $\Sigma
$-matrices is of $33$th order. Nowadays only binary semigroups up to $10$th
order are fully listed \cite{slo/plo,dis/jef/kel/kot}.
\end{example}

\subsection{The $n$-ary group of full $\mathit{\Sigma}$-matrices}

Let us consider the set of full $\mathit{\Sigma}$-matrices (\ref{sj}), which
is not closed under $n$-ary multiplication $\mu^{\left[  \mathbf{n}\right]  }$
(\ref{mnm}): the r.h.s. of their product contains the same multipliers $\pm1$,
$\pm i$ (see (\ref{skl})--(\ref{skl3}) for ternary case).Therefore, we should
enlarge this set, as for the $\sigma$-matrices (\ref{sq}).

\begin{definition}
The enlarged set of the $q$-generalized phase-shifted full $\mathit{\Sigma}%
$-matrices consists of $4q$ elements (see (\ref{is}), (\ref{q}) and
(\ref{sq}))%
\begin{equation}
\mathit{\hat{\Sigma}}_{j}^{\left(  q\right)  }\left(  r_{j}\right)
=e^{i\frac{2\pi}{q}r_{j}}\mathit{\Sigma}_{j},\ \ \ j=0,1,2,3,\ \ \ r_{j}%
=0,1,2,\ldots q-1,\ \ \ q\in\mathsf{Q}_{12}. \label{ski}%
\end{equation}

\end{definition}

\begin{assertion}
\label{asser-1}The set $\left\{  \mathit{\hat{\Sigma}}_{j}^{\left(  q\right)
}\left(  r\right)  \right\}  $ is closed with respect to $n$-ary
multiplication $\mu^{\left[  \mathbf{n}\right]  }$ (\ref{mnm}).
\end{assertion}

\begin{proof}
It follows from (\ref{sj}) that non-zero $n$-ary products of $\mathit{\Sigma}%
$-matrices are closed, as it is for the products of $n$ $q$-generalized
extended $\sigma$-matrices (\ref{sq}) which form a closed set, see
(\ref{mm1})--(\ref{mmn}).
\end{proof}

To construct an $n$-ary group we have to find the querelement (\ref{mq}%
)--(\ref{mw}) for full $\mathit{\Sigma}$-matrices.

\begin{proposition}
The querelement for the $q$-generalized phase-shifted full $\mathit{\Sigma}%
$-matrix (\ref{ski}) is%
\begin{equation}
\widetilde{\mathit{\Sigma}}_{j}^{\left(  q\right)  }\left(  r_{j}\right)
=e^{i\frac{2\pi}{q}\left(  r_{j}\dotplus_{q}\left(  -n\right)  r_{j}\right)
}\mathit{\Sigma}_{j}, \label{sq1}%
\end{equation}
where%
\begin{equation}
\widetilde{\mathit{\Sigma}}_{j}=\left\{
\begin{array}
[c]{c}%
\mathit{\Sigma}_{0},\ \text{if arity }n\text{ is even,}\\
\mathit{\Sigma}_{j},\ \text{if arity }n\text{ is odd.}%
\end{array}
\right.  \label{sq2}%
\end{equation}

\end{proposition}

\begin{proof}
We use the equation (\ref{mq}) for phase factors and the involutory (\ref{in})
to get (\ref{sq2}).
\end{proof}

\begin{proposition}
The set of the $q$-generalized phase-shifted $n$-ary full $\mathit{\Sigma}%
$-matrices forms $12$ non-isomorphic finite $n$-ary groups of order $4q$,
$q\in\mathsf{Q}_{12}$.
\end{proposition}

\begin{proof}
It follows from \textbf{Assertion} \ref{asser-1}, with the total polyadic
associativity following from the fact that the multiplication of ordinary
matrices is associative, and the querelements are given in (\ref{sq1}%
)--(\ref{sq2}).
\end{proof}

The formulas of $n$-ary multiplication for $q$-generalized full
$\mathit{\Sigma}$-matrices are cumbersome, but for some concrete cases can be
written manifestly.

\begin{example}
[\textsf{Ternary group of full} $\mathit{\Sigma}$\textsf{-matrices}]We derive
the multiplication of $q$-generalized phase-shifted ternary full
$\mathit{\Sigma}$-matrices from (\ref{skl})--(\ref{skl3}) by multiplying both
sides by three factors $e^{i\frac{2\pi}{q}r_{k}}e^{i\frac{2\pi}{q}r_{l}%
}e^{i\frac{2\pi}{q}r_{m}}$ and use (\ref{ek}). Thus, we get the $q$%
-generalized ternary products%
\begin{align}
&  \mu_{q}^{\left[  \mathbf{3}\right]  }\left[  \mathit{\hat{\Sigma}}%
_{k}^{\left(  q\right)  }\left(  r_{k}\right)  ,\mathit{\hat{\Sigma}}%
_{l}^{\left(  q\right)  }\left(  r_{l}\right)  ,\mathit{\hat{\Sigma}}%
_{m}^{\left(  q\right)  }\left(  r_{m}\right)  \right]  =\delta_{kl}%
\mathit{\hat{\Sigma}}_{m}^{\left(  q\right)  }\left(  r_{k}\dotplus_{q}%
r_{l}\dotplus_{q}r_{m}\right)  +\delta_{lm}\mathit{\hat{\Sigma}}_{k}^{\left(
q\right)  }\left(  r_{k}\dotplus_{q}r_{l}\dotplus_{q}r_{m}\right) \nonumber\\
&  +\delta_{km}\mathit{\hat{\Sigma}}_{l}^{\left(  q\right)  }\left(
r_{k}\dotplus_{q}r_{l}\dotplus_{q}r_{m}\dotplus_{q}\frac{q}{2}\right)
+\left\vert \epsilon_{klm}\right\vert \mathit{\hat{\Sigma}}_{0}^{\left(
q\right)  }\left(  r_{k}\dotplus_{q}r_{l}\dotplus_{q}r_{m}\dotplus_{q}\frac
{q}{4}\dotplus_{q}\frac{q}{4}\left(  1-\epsilon_{klm}\right)  \right)  .
\label{m3q}%
\end{align}
The querelement now satisfies ternary involutivity, which follows from the
binary version (\ref{in}) and (\ref{sq1})%
\begin{equation}
\widetilde{\mathit{\Sigma}}_{j}^{\left(  q\right)  }=\mathit{\Sigma}_{j}.
\end{equation}
The ternary total associativity of the multiplication (\ref{m3q}) follows from
the ordinary associativity of the matrix product (\ref{mnm}). Therefore, the
set of $q$-generalized phase-shifted ternary full $\mathit{\Sigma}$-matrices
form a finite ternary group of order $4q$%
\begin{equation}
\hat{G}_{\mathit{\Sigma}}^{\left[  \mathbf{3}\right]  \left(  q\right)
}=\left\langle \left\{  \mathit{\hat{\Sigma}}_{j}^{\left(  q\right)  }\left(
r\right)  \right\}  \mid\mu_{q}^{\left[  \mathbf{3}\right]  },\widetilde
{\left(  \_\right)  }\right\rangle ,\ \ \ \ q\in\mathsf{Q}_{12}.
\end{equation}

In the minimal case $q=4$, the group of ternary full $\mathit{\Sigma}%
$-matrices $\hat{G}_{\mathit{\Sigma}}^{\left(  q\right)  }$ is of $16$th
order, while the Cayley table contains $16^{3}=4096$ entries.
\end{example}

\section{\textsc{Heterogeneous full }$\Sigma$\textsc{-matrices}}

Let us consider a further extension of the full $\mathit{\Sigma}$-matrices:
each entry in (\ref{sj}) contains a different $\sigma$-matrix.

\begin{definition}
The heterogeneous $n$-ary full $\mathit{\Sigma}^{het}$-matrices are
multi-index cyclic shift matrices of the following form%
\begin{equation}
\mathit{\Sigma}^{het}=\mathit{\Sigma}_{j_{1}j_{2}\ldots j_{n-1}}^{het}=\left(
\begin{array}
[c]{ccccccc}%
\mathsf{0} & \sigma_{j_{1}} & \mathsf{0} & \ldots & \mathsf{0} & \ldots &
\mathsf{0}\\
\mathsf{0} & \mathsf{0} & \sigma_{j_{2}} & \ldots & \mathsf{0} & \ldots &
\mathsf{0}\\
\mathsf{0} & \mathsf{0} & \mathsf{0} & \ddots & \vdots & \ldots & \vdots\\
\mathsf{0} & \mathsf{0} & \mathsf{0} & \ldots & \sigma_{j_{k}} & \ldots &
\mathsf{0}\\
\vdots & \vdots & \vdots & \vdots & \vdots & \ddots & \vdots\\
\mathsf{0} & \mathsf{0} & \mathsf{0} & \ldots & \mathsf{0} & \ldots &
\sigma_{j_{n-2}}\\
\sigma_{j_{n-1}} & \mathsf{0} & \mathsf{0} & \ldots & \mathsf{0} & \ldots &
\mathsf{0}%
\end{array}
\right)  ,\ \ j_{1},j_{2},\ldots,j_{n-1}=0,1,2,3, \label{sjj}%
\end{equation}
and the $4$ full $\mathit{\Sigma}$-matrices (\ref{sj}) can be called homogeneous.
\end{definition}

The heterogeneous full $\mathit{\Sigma}^{het}$-matrices can be presented and
the sum of the elementary $\Sigma$-matrices as was done for the (homogeneous)
full $\mathit{\Sigma}$-matrices (\ref{ssj})%
\begin{equation}
\Sigma_{j_{1}}^{\left(  \mathbf{1}\right)  }+\Sigma_{j_{2}}^{\left(
\mathbf{2}\right)  }+\ldots+\Sigma_{j_{n-1}}^{\left(  \mathbf{n-1}\right)
}=\mathit{\Sigma}_{j_{1}j_{2}\ldots j_{n-1}}^{het},\ \ j=0,1,2,3. \label{sjh}%
\end{equation}

The number of heterogeneous full $\mathit{\Sigma}^{het}$-matrices is $\left(
n-1\right)  ^{4}$, which follows from the formula for arrangements with repetitions.

The exact multiplication can be written in examples for concrete arities,
mostly of the lowest ones.

\begin{example}
[\textsf{Ternary heterogeneous full} $\mathit{\Sigma}$\textsf{-matrices}]In
the lowest ternary ($n=3$) case $\mathit{\Sigma}^{het}$ (\ref{sjj}) has two
indices%
\begin{equation}
\mathit{\Sigma}_{j_{1}j_{2}}^{het}=\left(
\begin{array}
[c]{cc}%
0 & \sigma_{j_{1}}\\
\sigma_{j_{2}} & 0
\end{array}
\right)  ,\ \ j_{1},j_{2}=0,1,2,3. \label{sj3}%
\end{equation}
There are $2^{4}=16$ ternary heterogeneous full $\mathit{\Sigma}^{het}%
$-matrices. The non-derived ternary product (\ref{mnm}) of $\mathit{\Sigma
}^{het}$-matrices now has the form%
\begin{equation}
\mathit{\Sigma}_{j_{1}j_{2}}^{het}=\mu_{het}^{\left[  \mathbf{3}\right]
}\left[  \mathit{\Sigma}_{j_{1}^{\prime}j_{2}^{\prime}}^{het},\mathit{\Sigma
}_{j_{1}^{\prime\prime}j_{2}^{\prime\prime}}^{het},\mathit{\Sigma}%
_{j_{1}^{\prime\prime\prime}j_{2}^{\prime\prime\prime}}^{het}\right]  =\left(
\begin{array}
[c]{cc}%
0 & \sigma_{j_{1}^{\prime}}\sigma_{j_{2}^{\prime\prime}}\sigma_{j_{1}%
^{\prime\prime\prime}}\\
\sigma_{j_{2}^{\prime}}\sigma_{j_{1}^{\prime\prime}}\sigma_{j_{2}%
^{\prime\prime\prime}} & 0
\end{array}
\right)  ,\ j_{1},j_{2},j_{1}^{\prime},j_{2}^{\prime},j_{1}^{\prime\prime
},j_{2}^{\prime\prime},j_{1}^{\prime\prime\prime},j_{2}^{\prime\prime\prime
}=0,1,2,3, \label{sh}%
\end{equation}
and we can use the identities for $\sigma$-matrices (\ref{s1})--(\ref{s2}).
There are different cases, depending on the presence of $0$ among the indices
$j$'s (or identities among $\sigma$'s):

\begin{enumerate}
\item There are no $0$'s among all $j$'s, $j_{1}=k$ and $j_{2}=l$,
$k,l=1,2,3$. Then we use (\ref{s2}) to get%
\begin{align}
&  \mu_{het}^{\left[  \mathbf{3}\right]  }\left[  \mathit{\Sigma}_{k^{\prime
}l^{\prime}}^{het},\mathit{\Sigma}_{k^{\prime\prime}l^{\prime\prime}}%
^{het},\mathit{\Sigma}_{k^{\prime\prime\prime}l^{\prime\prime\prime}}%
^{het}\right] \nonumber\\
&  =\left(
\begin{array}
[c]{cc}%
0 & \delta_{k^{\prime}l^{\prime\prime}}\sigma_{k^{\prime\prime\prime}}%
-\delta_{k^{\prime}k^{\prime\prime\prime}}\sigma_{l^{\prime\prime}}%
+\delta_{l^{\prime\prime}k^{\prime\prime\prime}}\sigma_{k^{\prime}}%
+i\epsilon_{k^{\prime}l^{\prime\prime}k^{\prime\prime\prime}}\sigma_{0}\\
\delta_{l^{\prime}k^{\prime\prime}}\sigma_{l^{\prime\prime\prime}}%
-\delta_{l^{\prime}l^{\prime\prime\prime}}\sigma_{k^{\prime\prime}}%
+\delta_{k^{\prime\prime}l^{\prime\prime\prime}}\sigma_{l^{\prime}}%
+i\epsilon_{l^{\prime}k^{\prime\prime}l^{\prime\prime\prime}}\sigma_{0} & 0
\end{array}
\right)  . \label{m3a}%
\end{align}

\item There is one $0$ among $j_{1}^{\prime},j_{2}^{\prime\prime}%
,j_{1}^{\prime\prime\prime}$, it is not important which one, and thus we get
(e.g. for $j_{1}^{\prime\prime\prime}=0$)%
\begin{align}
&  \mu_{het}^{\left[  \mathbf{3}\right]  }\left[  \mathit{\Sigma}_{k^{\prime
}l^{\prime}}^{het},\mathit{\Sigma}_{k^{\prime\prime}l^{\prime\prime}}%
^{het},\mathit{\Sigma}_{0l^{\prime\prime\prime}}^{het}\right] \nonumber\\
&  =\left(
\begin{array}
[c]{cc}%
0 & \delta_{k^{\prime}l^{\prime\prime}}\sigma_{0}+i\epsilon_{k^{\prime
}l^{\prime\prime}m}\sigma_{m}\\
\delta_{l^{\prime}k^{\prime\prime}}\sigma_{l^{\prime\prime\prime}}%
-\delta_{l^{\prime}l^{\prime\prime\prime}}\sigma_{k^{\prime\prime}}%
+\delta_{k^{\prime\prime}l^{\prime\prime\prime}}\sigma_{l^{\prime}}%
+i\epsilon_{l^{\prime}k^{\prime\prime}l^{\prime\prime\prime}}\sigma_{0} & 0
\end{array}
\right)  ,\ m=1,2,3. \label{m3b}%
\end{align}
In the same way, if one $0$ is among $j_{2}^{\prime},j_{1}^{\prime\prime
},j_{2}^{\prime\prime\prime}$, then (e.g. for $j_{2}^{\prime\prime\prime}=0$)%
\begin{align}
&  \mu_{het}^{\left[  \mathbf{3}\right]  }\left[  \mathit{\Sigma}_{k^{\prime
}l^{\prime}}^{het},\mathit{\Sigma}_{k^{\prime\prime}l^{\prime\prime}}%
^{het},\mathit{\Sigma}_{k^{\prime\prime\prime}0}^{het}\right] \nonumber\\
&  =\left(
\begin{array}
[c]{cc}%
0 & \delta_{k^{\prime}l^{\prime\prime}}\sigma_{k^{\prime\prime\prime}}%
-\delta_{k^{\prime}k^{\prime\prime\prime}}\sigma_{l^{\prime\prime}}%
+\delta_{l^{\prime\prime}k^{\prime\prime\prime}}\sigma_{k^{\prime}}%
+i\epsilon_{k^{\prime}l^{\prime\prime}k^{\prime\prime\prime}}\sigma_{0}\\
\delta_{l^{\prime}k^{\prime\prime}}\sigma_{0}+i\epsilon_{l^{\prime}%
k^{\prime\prime}m}\sigma_{m} & 0
\end{array}
\right)  ,\ m=1,2,3. \label{m3c}%
\end{align}

\item There are two $0$'s in different entries of (\ref{sh}), e.g.
$j_{1}^{\prime\prime\prime}=0$ and $j_{2}^{\prime\prime\prime}=0$, then%
\begin{align}
&  \mu_{het}^{\left[  \mathbf{3}\right]  }\left[  \mathit{\Sigma}_{k^{\prime
}l^{\prime}}^{het},\mathit{\Sigma}_{k^{\prime\prime}l^{\prime\prime}}%
^{het},\mathit{\Sigma}_{00}^{het}\right] \nonumber\\
&  =\left(
\begin{array}
[c]{cc}%
0 & \delta_{k^{\prime}l^{\prime\prime}}\sigma_{0}+i\epsilon_{k^{\prime
}l^{\prime\prime}m}\sigma_{m}\\
\delta_{l^{\prime}k^{\prime\prime}}\sigma_{0}+i\epsilon_{l^{\prime}%
k^{\prime\prime}m}\sigma_{m} & 0
\end{array}
\right)  ,\ m=1,2,3. \label{m3d}%
\end{align}

\end{enumerate}
\end{example}

Note that each non-zero matrix entry in (\ref{m3a})--(\ref{m3d}) is
proportional to one $\sigma$-matrix multiplied by $\pm1$, $\pm i$, as in
(\ref{s1})--(\ref{s2}). However, these multipliers can be different for
different entries, as opposed to the case of homogeneous $\mathit{\Sigma}%
$-matrices, for which these multipliers coincide, and can be placed as common
coefficients before matrices (see (\ref{skl})--(\ref{skl3})). Therefore, the
heterogeneous full $\mathit{\Sigma}^{het}$-matrices of the form (\ref{sjj})
are not closed with respect to the $n$-ary multiplication (\ref{mnm}), as the
four ordinary $\sigma$-matrices (\ref{s1}). Nevertheless, the enlarged (in a
special way) phase-shifted heterogeneous full $\mathit{\Sigma}^{het}$-matrices
(analogous to (\ref{sq})) can be closed under $n$-ary multiplication.

\begin{definition}
The enlarged set of the $q$-generalized element-wise phase-shifted
heterogeneous full $\mathit{\Sigma}^{het}$-matrices have the form%
\begin{align}
\mathit{\Sigma}^{het\left(  q\right)  }  &  =\mathit{\Sigma}_{j_{1}j_{2}\ldots
j_{n-1}}^{het\left(  q\right)  }\left(  r_{j_{1}},r_{j_{2}},\ldots,r_{j_{n-1}%
}\right) \nonumber\\
&  =\left(
\begin{array}
[c]{ccccccc}%
\mathsf{0} & e^{i\frac{2\pi}{q}r_{j_{1}}}\sigma_{j_{1}} & \mathsf{0} & \ldots
& \mathsf{0} & \ldots & \mathsf{0}\\
\mathsf{0} & \mathsf{0} & e^{i\frac{2\pi}{q}r_{j_{2}}}\sigma_{j_{2}} & \ldots
& \mathsf{0} & \ldots & \mathsf{0}\\
\mathsf{0} & \mathsf{0} & \mathsf{0} & \ddots & \vdots & \ldots & \vdots\\
\mathsf{0} & \mathsf{0} & \mathsf{0} & \ldots & e^{i\frac{2\pi}{q}r_{j_{k}}%
}\sigma_{j_{k}} & \ldots & \mathsf{0}\\
\vdots & \vdots & \vdots & \vdots & \vdots & \ddots & \vdots\\
\mathsf{0} & \mathsf{0} & \mathsf{0} & \ldots & \mathsf{0} & \ldots &
e^{i\frac{2\pi}{q}r_{j_{n-2}}}\sigma_{j_{n-2}}\\
e^{i\frac{2\pi}{q}r_{j_{n-1}}}\sigma_{j_{n-1}} & \mathsf{0} & \mathsf{0} &
\ldots & \mathsf{0} & \ldots & \mathsf{0}%
\end{array}
\right)  ,\label{shq}\\
j_{k}  &  =0,1,2,3,\ \ \ \ \ \ r_{j_{k}}=0,1,2,\ldots q-1,\ \ \ \ \ \ q\in
\mathsf{Q}_{12},\nonumber
\end{align}
and this set consists of $\left(  4q\left(  n-1\right)  \right)  ^{4}$
elements (see (\ref{is}), (\ref{q}) and (\ref{sq})).
\end{definition}

\begin{assertion}
\label{asser-hq}The set $\left\{  \mathit{\Sigma}_{j_{1}j_{2}\ldots j_{n-1}%
}^{het\left(  q\right)  }\left(  r_{j_{1}},r_{j_{2}},\ldots,r_{j_{n-1}%
}\right)  \right\}  $ is closed with respect to the non-derived $n$-ary
multiplication $\mu^{\left[  \mathbf{n}\right]  }$ (\ref{mnm}).
\end{assertion}

\begin{proof}
This follows from the fact that the products of $q$-generalized extended
$\sigma$-matrices (\ref{sq}) form a closed set, see (\ref{mm1})--(\ref{mmn})
for general cyclic shift matrices.
\end{proof}

\begin{proposition}
The set of the $q$-generalized element-wise phase-shifted heterogeneous full
$\mathit{\Sigma}^{het}$-matrices (\ref{shq}) is a finite $n$-ary group of
order $\left(  4q\left(  n-1\right)  \right)  ^{4}$%
\begin{equation}
\hat{G}_{\mathit{\Sigma}}^{het\left[  \mathbf{n}\right]  \left(  q\right)
}=\left\langle \left\{  \mathit{\Sigma}_{j_{1}j_{2}\ldots j_{n-1}}^{het\left(
q\right)  }\left(  r_{j_{1}},r_{j_{2}},\ldots,r_{j_{n-1}}\right)  \right\}
\mid\mu_{het\left(  q\right)  }^{\left[  \mathbf{n}\right]  },\widetilde
{\left(  \_\right)  }\right\rangle ,\ \ \ \ q\in\mathsf{Q}_{12}. \label{ghq}%
\end{equation}

\end{proposition}

\begin{proof}
The set $\left\{  \mathit{\Sigma}_{j_{1}j_{2}\ldots j_{n-1}}^{het\left(
q\right)  }\left(  r_{j_{1}},r_{j_{2}},\ldots,r_{j_{n-1}}\right)  \right\}  $
is closed by \textbf{Assertion} \ref{asser-hq}. The total polyadic
associativity is governed by the associativity of ordinary matrix
multiplication. Each element has its querelement $\widetilde{\mathit{\Sigma
}^{het\left(  q\right)  }}$ defined by the general formula (\ref{mq}), where
the entries are of the form (\ref{shq}), because of (\ref{mw}) and the fact
that the product of $n$ ordinary $\sigma$-matrices is proportional to one
$\sigma$-matrix multiplied by the cyclic group factor.
\end{proof}

\begin{example}
[Ternary group of element-wise phase-shifted heterogeneous full
$\mathit{\Sigma}^{het}$-matrices]In the ternary $n=3$ case, for
$\mathit{\Sigma}^{het\left(  q\right)  }$ (\ref{shq}) we have%
\begin{equation}
\mathit{\Sigma}_{j_{1}j_{2}}^{het\left(  q\right)  }\left(  r_{j_{1}}%
,r_{j_{2}}\right)  =\left(
\begin{array}
[c]{cc}%
0 & e^{i\frac{2\pi}{q}r_{j_{1}}}\sigma_{j_{1}}\\
e^{i\frac{2\pi}{q}r_{j_{2}}}\sigma_{j_{2}} & 0
\end{array}
\right)  \ \ j_{1,2}=0,1,2,3,\ r_{j_{k}}=0,1,2,\ldots q-1,\ q\in
\mathsf{Q}_{12}. \label{shj}%
\end{equation}

The order of the ternary group $\hat{G}_{\mathit{\Sigma}}^{het\left[
\mathbf{3}\right]  \left(  q\right)  }$ is $4096q^{4}$, which for minimal case
$q=4$ becomes $1048\,576$. The ternary multiplication can be obtained from
(\ref{sh})--(\ref{m3d}) by adding the phase factors%
\begin{align}
&  \mathit{\Sigma}_{j_{1}j_{2}}^{het\left(  q\right)  }\left(  r_{j_{1}%
},r_{j_{2}}\right)  =\mu_{het\left(  q\right)  }^{\left[  \mathbf{3}\right]
}\left[  \mathit{\Sigma}_{j_{1}^{\prime}j_{2}^{\prime}}^{het}\left(
r_{j_{1}^{\prime}},r_{j_{2}^{\prime}}\right)  ,\mathit{\Sigma}_{j_{1}%
^{\prime\prime}j_{2}^{\prime\prime}}^{het}\left(  r_{j_{1}^{\prime\prime}%
},r_{j_{2}^{\prime\prime}}\right)  ,\mathit{\Sigma}_{j_{1}^{\prime\prime
\prime}j_{2}^{\prime\prime\prime}}^{het}\left(  r_{j_{1}^{\prime\prime\prime}%
},r_{j_{2}^{\prime\prime\prime}}\right)  \right] \nonumber\\
&  =\left(
\begin{array}
[c]{cc}%
0 & e^{i\frac{2\pi}{q}\left(  r_{j_{1}^{\prime}}+r_{j_{2}^{\prime\prime}%
}+r_{j_{1}^{\prime\prime\prime}}\right)  }\sigma_{j_{1}^{\prime}}\sigma
_{j_{2}^{\prime\prime}}\sigma_{j_{1}^{\prime\prime\prime}}\\
e^{i\frac{2\pi}{q}\left(  r_{j_{2}^{\prime}}+r_{j_{1}^{\prime\prime}}%
+r_{j_{2}^{\prime\prime\prime}}\right)  }\sigma_{j_{2}^{\prime}}\sigma
_{j_{1}^{\prime\prime}}\sigma_{j_{2}^{\prime\prime\prime}} & 0
\end{array}
\right)  ,\label{shqq}\\
&  j_{1},j_{2},j_{1}^{\prime},j_{2}^{\prime},j_{1}^{\prime\prime}%
,j_{2}^{\prime\prime},j_{1}^{\prime\prime\prime},j_{2}^{\prime\prime\prime
}=0,1,2,3,\ \ r_{j_{1,2}^{\prime}},r_{j_{1,2}^{\prime\prime}},r_{j_{1,2}%
^{\prime\prime\prime}}=0,1,2,\ldots q-1,\ q\in\mathsf{Q}_{12}.\nonumber
\end{align}
We consider $3$ cases as well and use (\ref{s1})--(\ref{s2}) with (\ref{ek})
to get

\begin{enumerate}
\item There are no $0$'s among all $j$'s, $j_{1}=k$ and $j_{2}=l$,
$k,l=1,2,3$. Then we use (\ref{s2}) to get

{\tiny
\begin{align}
&  \mu_{het\left(  q\right)  }^{\left[  \mathbf{3}\right]  }\left[
\mathit{\Sigma}_{k^{\prime}l^{\prime}}^{het\left(  q\right)  }\left(
r_{k^{\prime}},r_{l^{\prime}}\right)  ,\mathit{\Sigma}_{k^{\prime\prime
}l^{\prime\prime}}^{het\left(  q\right)  }\left(  r_{k^{\prime\prime}%
},r_{l^{\prime\prime}}\right)  ,\mathit{\Sigma}_{k^{\prime\prime\prime
}l^{\prime\prime\prime}}^{het\left(  q\right)  }\left(  r_{k^{\prime
\prime\prime}},r_{l^{\prime\prime\prime}}\right)  \right]  =\nonumber\\
&  \left(
\begin{array}
[c]{cc}%
0 &
\begin{array}
[c]{c}%
e^{i\frac{2\pi}{q}\left(  r_{k^{\prime}}+r_{l^{\prime\prime}}+r_{k^{\prime
\prime\prime}}\right)  \left(  \delta_{k^{\prime}l^{\prime\prime}}%
\sigma_{k^{\prime\prime\prime}}+\delta_{k^{\prime}k^{\prime\prime\prime}%
}e^{i\pi}\sigma_{l^{\prime\prime}}\right.  }\\
\left.  +\delta_{l^{\prime\prime}k^{\prime\prime\prime}}\sigma_{k^{\prime}%
}+\left\vert \epsilon_{k^{\prime}l^{\prime\prime}k^{\prime\prime\prime}%
}\right\vert e^{i\frac{\pi}{2}(2-\epsilon_{k^{\prime}l^{\prime\prime}%
k^{\prime\prime\prime}})}\sigma_{0}\right)
\end{array}
\\%
\begin{array}
[c]{c}%
e^{i\frac{2\pi}{q}\left(  r_{l^{\prime}}+r_{k^{\prime\prime}}+r_{l^{\prime
\prime\prime}}\right)  }\left(  \delta_{l^{\prime}k^{\prime\prime}}%
\sigma_{l^{\prime\prime\prime}}+\delta_{l^{\prime}l^{\prime\prime\prime}%
}e^{i\pi}\sigma_{k^{\prime\prime}}\right. \\
\left.  +\delta_{k^{\prime\prime}l^{\prime\prime\prime}}\sigma_{l^{\prime}%
}+\left\vert \epsilon_{l^{\prime}k^{\prime\prime}l^{\prime\prime\prime}%
}\right\vert e^{i\frac{\pi}{2}(2-\epsilon_{l^{\prime}k^{\prime\prime}%
l^{\prime\prime\prime}})}\sigma_{0}\right)
\end{array}
& 0
\end{array}
\right)  . \label{m3h1}%
\end{align}
}

\item There is one $0$ among $j_{1}^{\prime},j_{2}^{\prime\prime}%
,j_{1}^{\prime\prime\prime}$, it is not important which one, and thus we get
(e.g. for $j_{1}^{\prime\prime\prime}=0$){\tiny
\begin{align}
&  \mu_{het\left(  q\right)  }^{\left[  \mathbf{3}\right]  }\left[
\mathit{\Sigma}_{k^{\prime}l^{\prime}}^{het\left(  q\right)  }\left(
r_{k^{\prime}},r_{l^{\prime}}\right)  ,\mathit{\Sigma}_{k^{\prime\prime
}l^{\prime\prime}}^{het\left(  q\right)  }\left(  r_{k^{\prime\prime}%
},r_{l^{\prime\prime}}\right)  ,\mathit{\Sigma}_{0l^{\prime\prime\prime}%
}^{het\left(  q\right)  }\left(  r_{0},r_{l^{\prime\prime\prime}}\right)
\right]  =\nonumber\\
&  \left(
\begin{array}
[c]{cc}%
0 & e^{i\frac{2\pi}{q}\left(  r_{k^{\prime}}+r_{l^{\prime\prime}}%
+r_{0}\right)  }\left(  \delta_{k^{\prime}l^{\prime\prime}}\sigma
_{0}+\left\vert \epsilon_{k^{\prime}l^{\prime\prime}m}\right\vert
e^{i\frac{\pi}{2}(2-\epsilon_{k^{\prime}l^{\prime\prime}m})}\sigma_{m}\right)
\\%
\begin{array}
[c]{c}%
e^{i\frac{2\pi}{q}\left(  r_{l^{\prime}}+r_{k^{\prime\prime}}+r_{l^{\prime
\prime\prime}}\right)  \left(  \delta_{l^{\prime}k^{\prime\prime}}%
\sigma_{l^{\prime\prime\prime}}+\delta_{l^{\prime}l^{\prime\prime\prime}%
}e^{i\pi}\sigma_{k^{\prime\prime}}\right.  }\\
\left.  +\delta_{k^{\prime\prime}l^{\prime\prime\prime}}\sigma_{l^{\prime}%
}+\left\vert \epsilon_{l^{\prime}k^{\prime\prime}l^{\prime\prime\prime}%
}\right\vert e^{i\frac{\pi}{2}(2-\epsilon_{l^{\prime}k^{\prime\prime}%
l^{\prime\prime\prime}})}\sigma_{0}\right)
\end{array}
& 0
\end{array}
\right)  . \label{m3h2}%
\end{align}
} In the same way, if there is one $0$ among $j_{2}^{\prime},j_{1}%
^{\prime\prime},j_{2}^{\prime\prime\prime}$, then (e.g. for $j_{2}%
^{\prime\prime\prime}=0$){\tiny
\begin{align}
&  \mu_{het\left(  q\right)  }^{\left[  \mathbf{3}\right]  }\left[
\mathit{\Sigma}_{k^{\prime}l^{\prime}}^{het\left(  q\right)  }\left(
r_{k^{\prime}},r_{l^{\prime}}\right)  ,\mathit{\Sigma}_{k^{\prime\prime
}l^{\prime\prime}}^{het\left(  q\right)  }\left(  r_{k^{\prime\prime}%
},r_{l^{\prime\prime}}\right)  ,\mathit{\Sigma}_{k^{\prime\prime\prime}%
0}^{het\left(  q\right)  }\left(  r_{k^{\prime\prime\prime}},s_{0}\right)
\right]  =\nonumber\\
&  \left(
\begin{array}
[c]{cc}%
0 &
\begin{array}
[c]{c}%
e^{i\frac{2\pi}{q}\left(  r_{k^{\prime}}+r_{l^{\prime\prime}}+r_{k^{\prime
\prime\prime}}\right)  }\left(  \delta_{k^{\prime}l^{\prime\prime}}%
\sigma_{k^{\prime\prime\prime}}+\delta_{k^{\prime}k^{\prime\prime\prime}%
}e^{i\pi}\sigma_{l^{\prime\prime}}\right. \\
\left.  +\delta_{l^{\prime\prime}k^{\prime\prime\prime}}\sigma_{k^{\prime}%
}+\left\vert \epsilon_{k^{\prime}l^{\prime\prime}k^{\prime\prime\prime}%
}\right\vert e^{i\frac{\pi}{2}(1-\epsilon_{k^{\prime}l^{\prime\prime}%
k^{\prime\prime\prime}})}\sigma_{0}\right)
\end{array}
\\
e^{i\frac{2\pi}{q}\left(  r_{l^{\prime}}+r_{k^{\prime\prime}}+s_{0}\right)
}\left(  \delta_{l^{\prime}k^{\prime\prime}}\sigma_{0}+\left\vert
\epsilon_{l^{\prime}k^{\prime\prime}m}\right\vert e^{i\frac{\pi}{2}%
(1-\epsilon_{l^{\prime}k^{\prime\prime}m})}\sigma_{m}\right)  & 0
\end{array}
\right)  . \label{m3h3}%
\end{align}
}

\item There are two $0$'s in different entries of (\ref{sh}), e.g.
$j_{1}^{\prime\prime\prime}=0$ and $j_{2}^{\prime\prime\prime}=0$, then{\tiny
\begin{align}
&  \mu_{het\left(  q\right)  }^{\left[  \mathbf{3}\right]  }\left[
\mathit{\Sigma}_{k^{\prime}l^{\prime}}^{het\left(  q\right)  }\left(
r_{k^{\prime}},r_{l^{\prime}}\right)  ,\mathit{\Sigma}_{k^{\prime\prime
}l^{\prime\prime}}^{het\left(  q\right)  }\left(  r_{k^{\prime\prime}%
},r_{l^{\prime\prime}}\right)  ,\mathit{\Sigma}_{00}^{het\left(  q\right)
}\left(  r_{0},s_{0}\right)  \right]  =\nonumber\\
&  \left(
\begin{array}
[c]{cc}%
0 & e^{i\frac{2\pi}{q}\left(  r_{k^{\prime}}+r_{l^{\prime\prime}}%
+r_{0}\right)  }\left(  \delta_{k^{\prime}l^{\prime\prime}}\sigma
_{0}+\left\vert \epsilon_{k^{\prime}l^{\prime\prime}m}\right\vert
e^{i\frac{\pi}{2}(2-\epsilon_{k^{\prime}l^{\prime\prime}m})}\sigma_{m}\right)
\\
e^{i\frac{2\pi}{q}\left(  r_{l^{\prime}}+r_{k^{\prime\prime}}+s_{0}\right)
}\left(  \delta_{l^{\prime}k^{\prime\prime}}\sigma_{0}+\left\vert
\epsilon_{l^{\prime}k^{\prime\prime}m}\right\vert e^{i\frac{\pi}{2}%
(2-\epsilon_{l^{\prime}k^{\prime\prime}m})}\sigma_{m}\right)  & 0
\end{array}
\right)  ,\label{m3h4}\\
&  r_{0},s_{0}=0,1,2,\ldots q-1,\ m=1,2,3.\nonumber
\end{align}
}
\end{enumerate}

Each non-zero matrix entry in (\ref{m3h1})--(\ref{m3h4}) is equal to one
$\sigma$-matrix multiplied by a phase factor from $C_{q}$, that is, it has the
form (\ref{shj}), and therefore the ternary multiplication $\mu_{het\left(
q\right)  }^{\left[  \mathbf{3}\right]  }$ is closed and non-derived. Since in
the ternary case (only) the querelement (\ref{mq}) coincides with the inverse
matrix (see the note after (\ref{eid})), and $\sigma$-matrices are reflections
(\ref{in}), we have
\begin{equation}
\widetilde{\mathit{\Sigma}}_{j_{1}j_{2}}^{het\left(  q\right)  }\left(
r_{j_{1}},r_{j_{2}}\right)  =\left(
\begin{array}
[c]{cc}%
0 & e^{-i\frac{2\pi}{q}r_{j_{2}}}\sigma_{j_{2}}\\
e^{-i\frac{2\pi}{q}r_{j_{1}}}\sigma_{j_{1}} & 0
\end{array}
\right)  \ \ j_{1,2}=0,1,2,3,\ r_{j_{k}}=0,1,2,\ldots q-1,\ q\in
\mathsf{Q}_{12}.
\end{equation}
The ternary associativity follows from the associativity of ordinary matrix
multiplication. Therefore,%
\begin{equation}
\hat{G}_{\mathit{\Sigma}}^{het\left[  \mathbf{3}\right]  \left(  q\right)
}=\left\langle \left\{  \mathit{\Sigma}_{j_{1}j_{2}}^{het\left(  q\right)
}\left(  r_{j_{1}},r_{j_{2}}\right)  \right\}  \mid\mu_{het\left(  q\right)
}^{\left[  \mathbf{3}\right]  },\widetilde{\left(  \_\right)  }\right\rangle ,
\end{equation}
is a finite ternary group of element-wise phase-shifted heterogeneous full
$\mathit{\Sigma}^{het}$-matrices of order $4096q^{4}$.
\end{example}

\bigskip

\bigskip

\textbf{Acknowledgements}. The author is deeply grateful to Mike Hewitt,
Thomas Nordahl, Vladimir Tkach and Raimund Vogl for useful discussions and
valuable help.

%% file: Duplij_PolSigma-arxiv.bbl
\begin{thebibliography}{}

\bibitem[\protect\citeauthoryear{Bagarello, Bavuma, and
  Russo}{\textcolor{blue}{\sc Bagarello et~al.}}{2021}]{bag/bav/rus}
{\textcolor{blue}{\sc Bagarello, F., Y.~Bavuma, and F.~G. Russo}} (2021).
\newblock Topological decompositions of the {P}auli group and their influence
  on dynamical systems.
\newblock {\em Math. Phys. Anal. Geom.\/}~{\bf 24} (2), 1--20.

\bibitem[\protect\citeauthoryear{Ball, Centelles, and
  Huber}{\textcolor{blue}{\sc Ball et~al.}}{2023}]{bal/cen/hub}
{\textcolor{blue}{\sc Ball, S., A.~Centelles, and F.~Huber}} (2023).
\newblock Quantum error-correcting codes and their geometries.
\newblock {\em Ann. Inst. Henri Poincare D, Comb. Phys. Interact.\/}~{\bf 10}
  (2), 337--405.

\bibitem[\protect\citeauthoryear{Besche, Eick, and
  O'Brien}{\textcolor{blue}{\sc Besche et~al.}}{2002}]{bes/eic/bri}
{\textcolor{blue}{\sc Besche, H.~U., B.~Eick, and E.~O'Brien}} (2002).
\newblock A millennium project: constructing small groups.
\newblock {\em Int. J. Algebra Comput.\/}~{\bf 12} (05), 623--644.

\bibitem[\protect\citeauthoryear{Besche, Eick, and
  O’Brien}{\textcolor{blue}{\sc Besche et~al.}}{2001}]{bes/eic/bri1}
{\textcolor{blue}{\sc Besche, H.~U., B.~Eick, and E.~A. O’Brien}} (2001).
\newblock The groups of order at most 2000.
\newblock {\em Electron. Res. Announc. Amer. Math. Soc.\/}~{\bf 7} (1), 1--4.

\bibitem[\protect\citeauthoryear{Conway, Curtis, Norton, Parker, and
  Wilson}{\textcolor{blue}{\sc Conway et~al.}}{1986}]{con/cur/nor}
{\textcolor{blue}{\sc Conway, J.~H., R.~T. Curtis, S.~P. Norton, R.~A. Parker,
  and R.~A. Wilson}} (1986).
\newblock {\em Atlas of Finite Groups}.
\newblock Oxford: Oxford Univ. Press.

\bibitem[\protect\citeauthoryear{Distler, Jefferson, Kelsey, and
  Kotthoff}{\textcolor{blue}{\sc Distler et~al.}}{2012}]{dis/jef/kel/kot}
{\textcolor{blue}{\sc Distler, A., C.~Jefferson, T.~Kelsey, and L.~Kotthoff}}
  (2012).
\newblock The semigroups of order 10.
\newblock In: {\textcolor{blue}{\sc M.~Milano}} (Ed.), {\em Principles and
  Practice of Constraint Programming}, Vol. 7514 of {\em Lecture Notes in
  Computer Science}, Berlin-Heidelberg: Springer, pp.\  883--899.

\bibitem[\protect\citeauthoryear{Djordjevic}{\textcolor{blue}{\sc
  Djordjevic}}{2021}]{djordjevic}
{\textcolor{blue}{\sc Djordjevic, I.}} (2021).
\newblock {\em Quantum Information Processing, Quantum Computing, and Quantum
  Error Correction. {An} Engineering Approach\/} (2nd ed.).
\newblock Amsterdam: Elsevier/Academic Press.

\bibitem[\protect\citeauthoryear{D\"ornte}{\textcolor{blue}{\sc
  D\"ornte}}{1929}]{dor3}
{\textcolor{blue}{\sc D\"ornte, W.}} (1929).
\newblock Unterschungen \"uber einen verallgemeinerten {G}ruppenbegriff.
\newblock {\em Math. Z.\/}~{\bf 29}, 1--19.

\bibitem[\protect\citeauthoryear{Duplij}{\textcolor{blue}{\sc
  Duplij}}{2022}]{duplij2022}
{\textcolor{blue}{\sc Duplij, S.}} (2022).
\newblock {\em Polyadic Algebraic Structures}.
\newblock Bristol: IOP Publishing.

\bibitem[\protect\citeauthoryear{Feynman, Leighton, and
  Sands}{\textcolor{blue}{\sc Feynman et~al.}}{1965}]{fey/lei/san}
{\textcolor{blue}{\sc Feynman, R.~P., R.~B. Leighton, and M.~Sands}} (1965).
\newblock {\em The Feynman Lectures on Physics, Vol. 3: Quantum Mechanics}.
\newblock Reading, MA: Addison-Wesley.

\bibitem[\protect\citeauthoryear{Kibler}{\textcolor{blue}{\sc
  Kibler}}{2009}]{kib2009}
{\textcolor{blue}{\sc Kibler, M.~R.}} (2009).
\newblock An angular momentum approach to quadratic {F}ourier transform,
  {H}adamard matrices, {G}auss sums, mutually unbiased bases, the unitary group
  and the {P}auli group.
\newblock {\em J. Phys. A\/}~{\bf 42} (35), 353001.

\bibitem[\protect\citeauthoryear{Liboff}{\textcolor{blue}{\sc
  Liboff}}{2002}]{liboff}
{\textcolor{blue}{\sc Liboff, R.~L.}} (2002).
\newblock {\em Introductory Quantum Mechanics}.
\newblock Addison-Wesley.

\bibitem[\protect\citeauthoryear{Nielsen and Chuang}{\textcolor{blue}{\sc
  Nielsen and Chuang}}{2000}]{nie/chu}
{\textcolor{blue}{\sc Nielsen, M.~A. and I.~L. Chuang}} (2000).
\newblock {\em Quantum Computation and Quantum Information}.
\newblock Cambridge: Cambridge University Press.

\bibitem[\protect\citeauthoryear{Schiff}{\textcolor{blue}{\sc
  Schiff}}{1968}]{schiff}
{\textcolor{blue}{\sc Schiff, L.~I.}} (1968).
\newblock {\em Quantum Mechanics}.
\newblock New York: McGraw-Hill.

\bibitem[\protect\citeauthoryear{Sloane and Plouffe}{\textcolor{blue}{\sc
  Sloane and Plouffe}}{1995}]{slo/plo}
{\textcolor{blue}{\sc Sloane, N. J.~A. and S.~Plouffe}} (1995).
\newblock {\em The Encyclopedia of Integer Sequences}.
\newblock Academic Press.

\end{thebibliography}
